\documentclass[12pt]{amsart}     
\usepackage{amsmath,amssymb,amsthm,mathrsfs}     
\usepackage{latexsym}    
\usepackage{tikz}
\usepackage{hyperref}     
\usepackage{multirow}     
\usepackage{color,colordvi,graphicx}
\newtheorem{theorem}{Theorem}     
\numberwithin{theorem}{section}
\newtheorem{lemma}[theorem]{Lemma}     
\newtheorem{corollary}[theorem]{Corollary}     
\newtheorem{proposition}[theorem]{Proposition}     

\theoremstyle{definition}  
     
\newtheorem{example}[theorem]{Example}     
\newtheorem{remark}[theorem]{Remark}     

\newcounter{FNC}[page]
\def\fauxfootnote#1{{\addtocounter{FNC}{2}$\Magenta{^\fnsymbol{FNC}}$%
     \let\thefootnote\relax\footnotetext{\Magenta{$^\fnsymbol{FNC}$#1}}}}
\renewcommand{\P}{{\mathbb P}}
\newcommand{\C}{{\mathbb C}}
\newcommand{\R}{{\mathbb R}}
\newcommand{\Z}{{\mathbb Z}}

\newcommand{\calC}{{\mathcal C}}

\newcommand{\calI}{{\mathcal I}}
\newcommand{\calJ}{{\mathcal J}}
\newcommand{\calS}{{\mathcal S}}
\def\scrI{\mathscr{I}}

\newcommand{\frakm}{{\mathfrak m}}

\newcommand{\ini}{{\rm in}}

\DeclareMathOperator{\codim}{{\rm codim}}
\DeclareMathOperator{\depth}{{\rm depth}}

\DeclareMathOperator{\supp}{{\rm supp}}
\DeclareMathOperator{\ann}{{\rm ann}}
\DeclareMathOperator{\mult}{{\rm mult}}
\DeclareMathOperator{\indeg}{{\rm indeg}}
\DeclareMathOperator{\nd}{{\rm end}}
\DeclareMathOperator{\reg}{{\rm reg}}
\DeclareMathOperator{\kernel}{{\rm kernel}}
\DeclareMathOperator{\image}{{\rm image}}
\DeclareMathOperator{\HF}{{\it HF}}
\DeclareMathOperator{\HP}{{\it HP}}
\DeclareMathOperator{\Ext}{{\it Ext}}

\newcommand{\defcolor}[1]{\RoyalBlue{#1}}
\newcommand{\demph}[1]{\defcolor{{\sl #1}}}

\title{Semialgebraic Splines}
\author[M.~DiPasquale]{Michael DiPasquale}     
\address{Michael DiPasquale\\     
         Department of Mathematics\\     
         Oklahoma State University\\     
         Stillwater\\
         OK \ 74078-1058\\     
         USA}     
\email{midipasq@gmail.com}     
\urladdr{\url{http://math.okstate.edu/people/mdipasq/}}   
\author[F.~Sottile]{Frank Sottile}     
\address{Frank Sottile\\     
         Department of Mathematics\\     
         Texas A\&M University\\     
         College Station\\     
         Texas \ 77843\\     
         USA}     
\email{sottile@math.tamu.edu}     
\urladdr{\url{http://www.math.tamu.edu/~sottile}}   
\author[L.~Sun]{Lanyin Sun}     
\address{Lanyin Sun\\     
          School of Mathematical Sciences\\
          Dalian University of Technology\\
          Dalian 116024
          China}
\email{lanyinsun@mail.dlut.edu.cn}%\email{lanyin09@gmail.com}
\thanks{Research of Sottile supported in part by NSF grant DMS-1501370.}
\thanks{Research of Sun supported in part by the National Natural Science Foundation of China (Nos. 11290143, 11271060, 11401077) and Fundamental Research of Civil Aircraft (No. MJ-F-2012-04)}

\begin{document}     
     
\begin{abstract}     
 Semialgebraic splines are functions that are piecewise polynomial with respect to a cell decomposition 
 into sets defined by polynomial inequalities.
 We study bivariate semialgebraic splines, formulating spaces of semialgebraic splines in terms of graded
 modules.
 We compute the dimension of the space of splines with large degree in two extreme
 cases when the cell decomposition has a single interior vertex.
 First, when the forms defining the edges span a two-dimensional space of forms of degree $n$---then the
 curves they define meet in $n^2$ points in the complex projective plane.
 In the other extreme, the curves have distinct slopes at the vertex and do not simultaneously vanish at
 any other point.
 We also study examples of the Hilbert function and polynomial in cases of a single vertex where the curves do
 not satisfy either of these extremes.
\end{abstract}

\maketitle     
%%%%%%%%%%%%%%%%%%%%%%%%%%%%%%%%%%%%%%%%%%%%%%%%%%%%%%%%%%%%%%%%%%%%%%%%%%%%     
%     
\section{Introduction}\label{Sec:intro}     
A multivariate spline is a function on a domain in $\R^n$ that is piecewise a polynomial with respect
to a cell decomposition $\Delta$ of the domain. 
A fundamental question is to describe the vector space of splines on $\Delta$ that have a
given smoothness and whose polynomial constituents have at most a fixed degree.
Traditionally, $\Delta$ is a simplicial~\cite{Strang} or polyhedral~\cite{Schum84} complex.
Here, we consider the case when $\Delta$ is a planar complex whose cells are bounded by arcs of algebraic
curves.
We will call splines on $\Delta$ \demph{semialgebraic splines}, as the cells are semialgebraic sets.

Wang made the first steps in semialgebraic splines~\cite{Wang75,Wang85}, observing that
smoothness is equivalent to the usual existence of smoothing cofactors across 1-cells satisfying
conformality conditions at each vertex.
Stiller~\cite{Stiller83} used sheaf cohomology to determine the dimensions of spline spaces in some cases when
$\Delta$ has a single interior vertex.
When $\Delta$ is a polyhedral complex, classical spline spaces were recast in terms of graded modules and
homological algebra by Billera~\cite{Billera88}, who further developed this 
with Rose~\cite{BiRo91,BiRo92} and there is further foundational work by Schenck and
Stillman~\cite{ShSt97a,ShSt97b}.  
We study %the more general 
semialgebraic splines %, concentrating on the case 
when $\Delta$ has a single interior vertex.  
In many cases we compute the Hilbert polynomial, which gives the dimensions of the spline spaces when the degree is
greater than the postulation number, which we also consider. 

In Section~\ref{S:SplineModules}, we fix our notation and give background on spline modules.
We treat the local case when the subdivision $\Delta$ has a single interior vertex $\upsilon$ in the next two 
sections.
In Section~\ref{S:pencil}, the forms defining the curves lie in pencil, so that they define a
scheme of degree $n^2$ in $\C\P^2$, where $n$ is the degree of each curve.
In Section~\ref{S:singleVertex}, the curves are smooth at $\upsilon$ and their only common zero is $\upsilon$.  
Some of this is similar to Stiller's work~\cite{Stiller83}, but our main results involve
hypotheses that are complementary and less restrictive than his (see Remark~\ref{R:Stiller}).
Both Sections~\ref{S:pencil} and~\ref{S:singleVertex} address the dimension of the spline space in large degree.  
In Section~\ref{S:Regularity} we show how results from the theory of linkage can be used to evaluate the dimension
of the spline space in low degree in some instances, and address the question of how large the degree must be for
the formulas of Section~\ref{S:pencil} and~\ref{S:singleVertex} to hold using Castelnuovo-Mumford regularity.  
We close with Section~\ref{S:examples} where we give examples that suggest some extensions of this work
when $\Delta$ has a single interior  vertex.

%%%%%%%%%%%%%%%%%%%%%%%%%%%%%%%%%%%%%%%%%%%%%%%%%%%%%%%%%%%%%%%%%%%%%%%%%%%%%%%%%
\section{Spline Modules}\label{S:SplineModules}

Billera~\cite{Billera88} introduced methods from homological algebra into the study of splines.
This was refined by Billera and Rose~\cite{BiRo91,BiRo92} and by Schenck and Stillman~\cite{ShSt97a,ShSt97b}, who
viewed spaces of splines as homogeneous summands of graded modules over the polynomial ring, so that
the dimension of spline spaces is given by the Hilbert function of the module.
We fix our notation and make the straightforward observation that this homological approach carries over
to semialgebraic splines, in the same spirit as Wang's observation that smoothing cofactors and
conformality conditions for polyhedral splines carry over to semialgebraic splines~\cite{Wang75,Wang85}.
For more complete background, we recommend \S~8.3 of~\cite{CLO05}.
Background concerning free resolutions and modules may be found in~\cite{CLO05} or~\cite{Eisenbud}.

Let $\Delta$ be a finite cell complex in the plane $\R^2$, whose 1-cells are arcs of irreducible real
algebraic curves. 
We call the 2-cells of $\Delta$, \demph{faces}, the 1-cells, \demph{edges}, and 0-cells,
\demph{vertices}. 
We assume that each vertex and edge of $\Delta$ lies in the boundary of some face (it is \demph{pure}), that it is
connected, and that it is \demph{hereditary}: for any faces $\sigma,\sigma'$ sharing a vertex $\upsilon$,
there is a sequence $\sigma=\sigma_0,\sigma_1\dotsc,\sigma_n=\sigma'$ of faces containing $\upsilon$ such that   
each pair $\sigma_{i-1},\sigma_i$ for $i=1,\dotsc,n$ shares an edge.
Write $\defcolor{|\Delta|}\subset\R^2$ for the support of $\Delta$.
We assume that $|\Delta|$ is contractible and require that each connected component of the intersection of 
two cells of $\Delta$ is a cell of $\Delta$.
Write $\Delta^\circ_i$ for the set of $i$-cells of $\Delta$ that lie in the interior of $|\Delta|$.
Every face $\sigma$ of $\Delta$ inherits the orientation of $\R^2$ and we fix an orientation of
each edge $\tau\in\Delta^\circ_1$.

Figure~\ref{F:cell_complex} shows a cell complex with one interior vertex, three interior edges (oriented
inwards) and three faces.
%%%%%%%%%%%%%%%%%%%%%%%%%%%%%%%%%%%%%%%%%%%%%%%%%%%%%%%%%%%%%%%%%%%%%%%%%%%%%%%%%
\begin{figure}[htb]

\includegraphics[height=80pt]{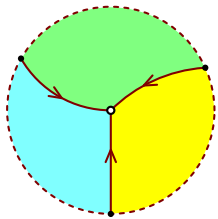}

 \caption{Cell complex with one interior vertex.}
 \label{F:cell_complex}
\end{figure}
%%%%%%%%%%%%%%%%%%%%%%%%%%%%%%%%%%%%%%%%%%%%%%%%%%%%%%%%%%%%%%%%%%%%%%%%%%%%%%%%%
Placing that vertex at the origin, $|\Delta|$ is the unit disc, and its edges (in
clockwise order) lie along the negative $y$-axis, the circle of radius 1 centered at
$(0,1)$, and the circle of radius $\sqrt{2}$ centered at $(1,-1)$. 

Let $R$ be a ring.
A \demph{chain complex $\calC$} is a sequence $C_0,C_1,\dotsc,C_n$ of $R$-modules with $R$-module maps
$\partial_i\colon C_i\to C_{i-1}$, whose compositions vanish, $\partial_{i-1}\circ\partial_i=0$,
so that the kernel of $\partial_{i-1}$ contains the image of $\partial_i$.
(Here, $C_{-1}=C_{n+1}=0$.)
The \demph{homology} of $\calC$ is the sequence of $R$-modules 
$\defcolor{H_i(\calC)}:=\kernel(\partial_{i-1})/\image(\partial_i)$, for
$i=0,\dotsc,n$. 

Let \defcolor{$R(\Delta)$} be the chain complex whose $i$th module has a basis given by the cells of
$\Delta^\circ_i$ and whose maps are induced by the boundary maps on the cells.
For the cell complex $\Delta$ of Figure~\ref{F:cell_complex}, $R(\Delta)$ is $R^3\to R^3\to R$.
Since the interior cells subdivide $|\Delta|$ with its boundary removed, the homology of the chain complex
$R(\Delta)$ is the relative homology $H_i(|\Delta|,\partial|\Delta|;R)$.
This always vanishes when $i=0$.
If $|\Delta|$ is connected and contractible, then we also have that $H_1(R(\Delta))=0$ and $H_2(R(\Delta))=R$.  

For integers $r,d\geq 0$, let $\defcolor{\widetilde{C}^r_d(\Delta)}$ be the real vector space of functions $f$
on $|\Delta|$ which have continuous $r$th order partial derivatives and whose
restriction to each face $\sigma$ of $\Delta$ is a polynomial $f_\sigma$ of degree at most $d$.
By~\cite{Wang75} (see also~\cite[Cor.~1.3]{BiRo92}), elements $f\in \widetilde{C}^r_d(\Delta)$ are lists
$(f_\sigma\mid\sigma\in\Delta_2)$ of polynomials such that if $\tau\in\Delta^\circ_1$ is an interior edge with
defining equation $g_\tau(x,y)=0$ that borders the two-dimensional faces $\sigma,\sigma'$, then   
$g_\tau^{r+1}$ divides the difference $f_\sigma-f_{\sigma'}$.
(The quotient is the smoothing cofactor at $\tau$.)

Figure~\ref{F:spline_graphs} displays the graphs of two splines on the complex $\Delta$ of 
%%%%%%%%%%%%%%%%%%%%%%%%%%%%%%%%%%%%%%%%%%%%%%%%%%%%%%%%%%%%%%%%%%%%%%%%%%%%%%%%%
\begin{figure}[htb]

\includegraphics[height=150pt]{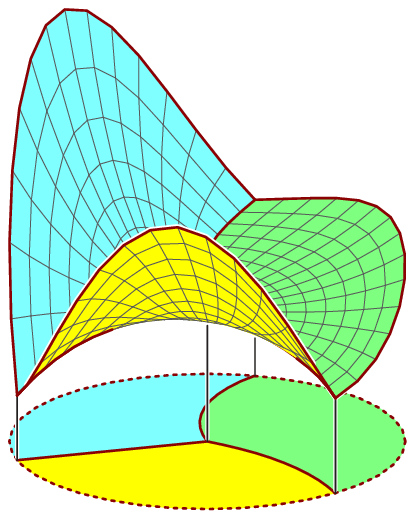}
\qquad
\includegraphics[height=150pt]{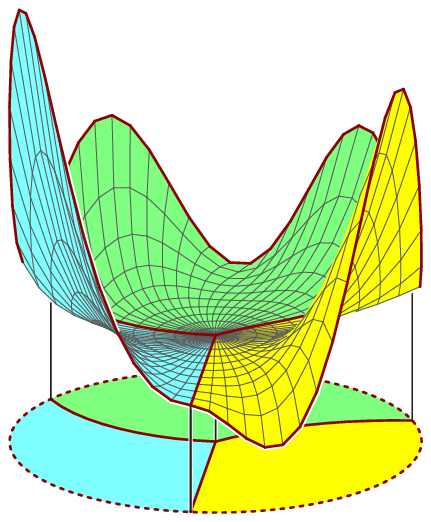}

\caption{Graphs of splines.}
\label{F:spline_graphs}
\end{figure}
%%%%%%%%%%%%%%%%%%%%%%%%%%%%%%%%%%%%%%%%%%%%%%%%%%%%%%%%%%%%%%%%%%%%%%%%%%%%%%%%%
Figure~\ref{F:cell_complex}.
The spline on the left lies in $\widetilde{C}^0_3(\Delta)$ and that on the right lies in
$\widetilde{C}^1_6(\Delta)$.
These are nonconstant splines on $\Delta$ of lowest degree for the given smoothness. 

Billera and Rose~\cite{BiRo91} observed that homogenizing spline spaces enables a global homological
approach to computing them. 
Let $\defcolor{S}:=\R[x,y,z]$ be the homogeneous coordinate ring of $\P^2(\R)$.
Write $\langle G_1,\dotsc,G_t\rangle$ for the ideal generated by $G_1,\dotsc,G_t$.
Let $\defcolor{C^r_d(\Delta)}$ be the vector space of lists $(F_\sigma \mid \sigma\in\Delta_2)$ of
homogeneous forms in $S$ of degree $d$ such that if $f_\sigma:= F_\sigma(x,y,1)$ is the dehomogenization of
$F_\sigma$, then $(f_\sigma \mid \sigma\in\Delta_2)\in \widetilde{C}^r_d(\Delta)$.
Define $\defcolor{C^r(\Delta)}:=\bigoplus_d C^r_d(\Delta)$ to be the direct sum of these homogenized spline spaces.
Call $C^r(\Delta)$ the \demph{spline module}.
It is a graded module of the graded ring $S$.

%%%%%%%%%%%%%%%%%%%%%%%%%%%%%%%%%%%%%%%%%%%%%%%%%%%%%%%%%%%%%%%%%%%%%%%%%%%%%%%%%
\begin{lemma}\label{L:SplineModule}
 The spline module $C^r(\Delta)$ is finitely generated.
 It is the kernel of the map
 \begin{equation}\label{Eq:spline_kernel}
   S^{\Delta_2}\ \simeq\ \bigoplus_{\sigma\in\Delta_2} S\ \xrightarrow{\ \ \partial_1\ \ }
   \bigoplus_{\tau\in\Delta^\circ_1} S/\langle G_\tau^{r+1}\rangle\ ,
 \end{equation}
 where $G_\tau$ is the homogeneous form defining the edge $\tau$ and if 
 $F=(F_\sigma \mid \sigma\in\Delta_2)\in S^{\Delta_2}$ 
 and $\tau\in\Delta^\circ_1$, then the $\tau$-component of $\partial F$ is the
 difference $F_\sigma-F_{\sigma'}$, where $\tau$ is a component of the intersection $\sigma\cap\sigma'$
 and its the orientation agrees with that induced from $\sigma$,
 but is opposite to that induced from $\sigma'$.
\end{lemma}
%%%%%%%%%%%%%%%%%%%%%%%%%%%%%%%%%%%%%%%%%%%%%%%%%%%%%%%%%%%%%%%%%%%%%%%%%%%%%%%%%

Let $M=\bigoplus_d M_d$ be a finitely generated graded $S$-module.
The \demph{Hilbert function} of $M$ records the dimensions of its graded pieces, 
$\defcolor{\HF(M,d)}:=\dim_\R M_d$.
%Commutative algebra~\cite{Eisenbud} informs us that t
There is an integer $d_0\geq 0$ such that if $d>d_0$, then the Hilbert function is a polynomial, called the
\demph{Hilbert polynomial} of $M$, \defcolor{$\HP(M,d)$}~\cite{Eisenbud}.
The \demph{postulation number} of $M$ is the minimal such $d_0$, the greatest integer at which the Hilbert
function and Hilbert polynomial disagree. 
The reason for these definitions is that the problem of computing the dimensions $\dim C^r_d(\Delta)$ of the
spline spaces is equivalent to computing the Hilbert function of the spline module $C^r(\Delta)$, which equals its
Hilbert polynomial for $d>d_0$.

Table~\ref{T:HilbertFunction} gives the Hilbert function and Hilbert polynomial of $C^r(\Delta)$ for
$r=0,\dotsc,3$, where $\Delta$ is the cell complex of Figure~\ref{F:cell_complex}.  
The polynomials may be verified using Theorem~\ref{Cor:generalPowers}.
%%%%%%%%%%%%%%%%%%%%%%%%%%%%%%%%%%%%%%%%%%%%%%%%%%%%%%%%%%%%%%%%%%%%%%%%%%%%%%%%%
\begin{table}[htb]
 \caption{Hilbert function and polynomial of $C^r(\Delta)$.}
 \label{T:HilbertFunction}
  \begin{tabular}{|r|rrrrrrrrrrrrrr|c|r|}
    \hline  
    $r\backslash d$\rule{0pt}{12pt}
       &0&1&2&3&4&  5&  6&  7&  8&   9&  10&  11&  12&  13&Polynomial & $d_0$\\\hline
    $0$&1&3&6&\defcolor{13}& 23 & 36& 52& 71& 93& 118& 146& 177& 211& 248&\raisebox{-6pt}{\rule{0pt}{18pt}}%
                         $\frac{3}{2}d^2{-}\frac{1}{2}d{+}1$& 1\\
    $1$&1&3&6&10&15& 21&\defcolor{30}& 44& 61&  81& 104& 130& 159& 191&\raisebox{-6pt}{\rule{0pt}{18pt}}%
                         $\frac{3}{2}d^2{-}\frac{11}{2}d{+}9$&5 \\
    $2$&1&3&6&10&15& 21& 28& 36& 45& \defcolor{57}&  73&  94& 118& 145&\raisebox{-6pt}{\rule{0pt}{18pt}}%
                         $\frac{3}{2}d^2{-}\frac{21}{2}d{+}28$&9 \\
    $3$&1&3&6&10&15& 21& 28& 36& 45&  55&  66&  78&\defcolor{93}& 111&\raisebox{-6pt}{\rule{0pt}{18pt}}%
                      %134 162 193
                         $\frac{3}{2}d^2{-}\frac{31}{2}d{+}57$&13 \\\hline
% \binom{d+2}{2}   1, 3, 6, 10, 15, 21, 28, 36, 45, 55, 66, 78, 91, 105, 120, 136, 153
%    3*d^2/2 - (10*r+1)*d/2 + (1,9,28,57,97,147,208,279,361,453,556,669,793,927)
%                              0 1  2  3  4   5   6   7   8   9  10  11  12  13
% dim C^r_d(\Delta)  3*d^2/2 - (10*r+1)*d/2 + c(r), where c(r) is:
% r even  c(r) = 21*r^2/4 + 3*r +1:
% r odd   c(r) = 21*r^2/4 + 3*r + 3/4:
   $\tbinom{d+2}{2}$&1&3&6&10&15&21&28&36&45&55&66&78&91&105&\raisebox{-6pt}{\rule{0pt}{18pt}}%
     $\frac{1}{2}d^2+\frac{1}{2}d$& 0 \\\hline
 \end{tabular}
\end{table}
%%%%%%%%%%%%%%%%%%%%%%%%%%%%%%%%%%%%%%%%%%%%%%%%%%%%%%%%%%%%%%%%%%%%%%%%%%%%%%%%%
Its last row is the Hilbert function/polynomial of $\R[x,y,z]$, these are the constant splines---splines that
are restrictions of polynomials on $\R^2$. 
There are only constant splines in degrees less than $3r+3$.
The last column is the postulation number.
%\Red{{\sf The postulation numbers start out: $1,5,9,13,17,20,24,28,32,35,39,43,47,50,54,58$ (these agree with
%    $\reg(S/I)$ where $I=\langle f^{r+1},g^{r+1},h^{r+1} \rangle$ and $f,g,h$ are the three forms defining
%    the line and the two quadrics).}} 

For $\tau\in\Delta^\circ_1$, define $\defcolor{J(\tau)}:=\langle G_\tau^{r+1}\rangle$, the principal ideal
generated by $G_\tau^{r+1}$ and for $\upsilon\in\Delta^\circ_0$, define \defcolor{$J(\upsilon)$} to be the
ideal generated by all $J(\tau)$ where $\tau$ is incident on $\upsilon$.
Let $\calJ_1$ and $\calJ_0$ be the direct sums of these ideals, 
\[
\calJ_1\ :=\ \bigoplus_{\tau\in\Delta^\circ_1} J(\tau)
\qquad \mbox{and}\qquad    
\calJ_0\ :=\ \bigoplus_{\upsilon\in\Delta^\circ_0} J(\upsilon)\,.
\]
Then $\calJ \colon\calJ_1\xrightarrow{\partial_1}\calJ_0$ is a complex of $S$-modules, with $\partial_1$
the obvious map.
This is a subcomplex of the chain complex $\defcolor{\calS}:=S(\Delta)$ that computes the homology of
the pair $H_*(|\Delta|,\partial|\Delta|;S)$.
We have the short exact sequence of complexes of $S$-modules,
 \begin{equation}\label{Eq:SES}
   0\ \longrightarrow\ \calJ\ \longrightarrow\ 
   \calS\ \longrightarrow\ \calS/\calJ\ \longrightarrow\  0\,,
 \end{equation}
where $\calS/\calJ$ is the quotient complex,
 \[
  0\ \longrightarrow\ 
  \bigoplus_{\sigma\in\Delta_2} S\ \xrightarrow{\ \partial_2\ }\ 
  \bigoplus_{\tau\in\Delta^\circ_1} S/J(\tau)\ \xrightarrow{\ \partial_1\ }\ 
  \bigoplus_{\upsilon\in\Delta^\circ_0} S/J(\upsilon)\ \longrightarrow\ 0.
 \]
Observe that $C^r(\Delta)$ is the kernel of $\partial_2$.
That is, $C^r(\Delta)=H_2(\calS/\calJ)$.
The short exact sequence~\eqref{Eq:SES} gives the long exact sequence in homology
(note that $H_2(\calJ)=0$).
 \begin{multline*}
   \qquad \;0\to H_2(\calS)\to H_2(\calS/\calJ)\to H_1(\calJ)\to 
    H_1(\calS) \\
   \to H_1(\calS/\calJ)\to H_0(\calJ)\to H_0(\calS)\to H_0(\calS/\calJ)\to 0\,.\ \qquad
 \end{multline*}
%

%%%%%%%%%%%%%%%%%%%%%%%%%%%%%%%%%%%%%%%%%%%%%%%%%%%%%%%%%%%%%%%%%%%%%%%%%%%%%%%%%
\begin{proposition}\label{P:exact_sequences}
 We have $H_0(\calS/\calJ)=0$.
 If the support $|\Delta|$ of $\Delta$ is contractible, then $H_1(\calS/\calJ)\simeq H_0(\calJ)$ and 
 $C^r(\Delta)\simeq S\oplus H_1(\calJ)$, with the factor of $S$ the constant splines.

 If there is a unique interior vertex $\upsilon$, then $0=H_1(\calS/\calJ)=H_0(\calJ)$ and $H_1(\calJ)$ is the
 module of syzygies on the list of forms $(G_\tau^{r+1}\mid \tau\in\Delta^\circ_1)$.
\end{proposition}
%%%%%%%%%%%%%%%%%%%%%%%%%%%%%%%%%%%%%%%%%%%%%%%%%%%%%%%%%%%%%%%%%%%%%%%%%%%%%%%%%

%%%%%%%%%%%%%%%%%%%%%%%%%%%%%%%%%%%%%%%%%%%%%%%%%%%%%%%%%%%%%%%%%%%%%%%%%%%%%%%%%
\begin{proof}
 Since $\calS$ is the complex $S(\Delta)$, $H_0(\calS)=0$ so that $H_0(\calS/\calJ)=0$.
 If $|\Delta|$ is contractible, then $H_1(\calS)=0$ and $H_2(\calS)=S$. 
 Thus the remaining long exact sequence splits into sequences of lengths 2 and 3.
 The first gives $H_1(\calS/\calJ)\simeq H_0(\calJ)$ and the second is
\[
   0\ \longrightarrow\ S\ \longrightarrow\ C^r(\Delta)\ \longrightarrow\ H_1(\calJ)\ \longrightarrow\ 0\,,
\]
 giving the direct sum decomposition $C^r(\Delta)\simeq S\oplus H_1(\calJ)$, as the first map has a splitting 
 $(F_\sigma\mid\sigma\in\Delta_2)\mapsto F_{\sigma_0}$ given by any $\sigma_0\in\Delta_2$.
 The kernel $H_2(\calS)$ of the map $\partial_2$ of $\calS$ is the submodule of constant splines.

 Lastly, if there is a unique interior vertex $\upsilon$, then the forms 
 $\{G_\tau^{r+1}\mid \tau\in\Delta^\circ_1\}$ generate $J(\upsilon)$.
 Thus $H_0(\calJ)=0$ and the complex $\calJ$ is the first step in the
 resolution of the ideal $J(\upsilon)$ given the generators $( G_\tau^{r+1}\mid \tau\in\Delta^\circ_1)$.
 It follows that $H_1(\calJ)$ is the module of syzygies (or relations) on the forms 
 $(G_\tau^{r+1}\mid\tau\in\Delta^\circ_1)$.   
 When $J(\upsilon)$ is minimally generated by $(G_\tau^{r+1}\mid \tau\in\Delta^\circ_1)$, then 
 $H_1(\calJ)\simeq \mbox{syz}(J(\upsilon))$.
\end{proof}
%%%%%%%%%%%%%%%%%%%%%%%%%%%%%%%%%%%%%%%%%%%%%%%%%%%%%%%%%%%%%%%%%%%%%%%%%%%%%%%%%

Write \defcolor{$\phi_2$} for the number of faces of $\Delta$, \defcolor{$\phi_1$} for the number of interior edges,
and \defcolor{$\phi_0$} for the number of interior vertices, and for an interior edge $\tau\in\Delta^\circ_1$, let
\defcolor{$n_\tau$} be the degree of the form $G_\tau$ defining $\tau$.

%%%%%%%%%%%%%%%%%%%%%%%%%%%%%%%%%%%%%%%%%%%%%%%%%%%%%%%%%%%%%%%%%%%%%%%%%%%%%%%%%
\begin{corollary}\label{C:DimFormula}
 Suppose that the support $|\Delta|$ of $\Delta$ is contractible.
 Then for $r$ and $d$,
 \begin{equation}\label{Eq:dimFormula}
   \dim C^r_d(\Delta)\ =\ 
    (\phi_2-\phi_1)\tbinom{d+2}{2} + \sum_{\tau\in\Delta^\circ_1} \tbinom{d-(r+1)n_\tau+2}{2}
     + \sum_{\upsilon\in\Delta^\circ_0}\dim(S/J(\upsilon))_d
     + \dim H_0(\calJ)_d\,.
 \end{equation}
 When $\Delta$ has a unique interior vertex $\upsilon$, we have 
 \begin{equation}\label{Eq:DimensionFormula}
   \dim C^r_d(\Delta)\ =\ 
   \sum_{\tau\in\Delta^\circ_1} \tbinom{d-(r+1)n_\tau+2}{2}
     + \dim(S/J(\upsilon))_d\,.
 \end{equation}
 For $d\gg 0$, $\dim(S/J(\upsilon))_d$ is the degree of the scheme defined by $J(\upsilon)$. 
\end{corollary}
%%%%%%%%%%%%%%%%%%%%%%%%%%%%%%%%%%%%%%%%%%%%%%%%%%%%%%%%%%%%%%%%%%%%%%%%%%%%%%%%%

Formula~\eqref{Eq:DimensionFormula} is~\cite[Cor.~3.2]{Stiller83}, which is for mixed splines (see
Remark~\ref{R:mixed}). 
Recall that $S(-a)$ is the free $S$-module with one generator of degree $a$.

%%%%%%%%%%%%%%%%%%%%%%%%%%%%%%%%%%%%%%%%%%%%%%%%%%%%%%%%%%%%%%%%%%%%%%%%%%%%%%%%%
\begin{proof}
 From the complex $\calS/\calJ$, we have
 \begin{multline*}
  \qquad\HF(H_2(\calS/\calJ),d)\ -\ \HF(H_1(\calS/\calJ),d)\ +\ \HF(H_0(\calS/\calJ),d)\ =\ \\
  \HF((\calS/\calJ)_2,d)\ -\  \HF((\calS/\calJ)_1,d)\ +\  \HF((\calS/\calJ)_0,d)\,.\qquad
 \end{multline*}
 As $|\Delta|$ is contractible, $H_0(\calS/\calJ)=0$ and $H_1(\calS/\calJ)\simeq H_0(\calJ)$.
 Since $(\calS/\calJ)_2\simeq S^{\Delta_2}$, its Hilbert function is
 $\phi_2\binom{d+2}{2}$.
 From the sum of short exact sequences defining $(\calS/\calJ)_1$, 
\[
   \bigoplus_{\tau\in\Delta^\circ_1} \Bigl( S(-(r{+}1)n_\tau)\ 
     \xrightarrow{\ \cdot G_\tau^{r+1}\ }\ S\ \longrightarrow\ 
            S/\langle G_\tau^{r+1}\rangle\ =\  S/J(\tau)\Bigr)\ ,
\]
 we have that 
\[
  \HF((\calS/\calJ)_1,d)\ =\ \phi_1\tbinom{d+2}{2}\ -\ 
    \sum_{\tau\in\Delta^\circ_1}\tbinom{d-(r{+}1)n_\tau+2}{2}\ .
\]
 As $C^r(\Delta)=H_2(\calS/\calJ)$, and 
 $(\calS/\calJ)_0=\bigoplus_{\upsilon\in\Delta_0} S/J(\upsilon)$, this implies formula~\eqref{Eq:dimFormula}. 

 When $\Delta$ has a unique interior vertex $\upsilon$, $H_0(\calJ)=0$ and $\phi_2=\phi_1$,
 giving~\eqref{Eq:DimensionFormula}. 
\end{proof}
%%%%%%%%%%%%%%%%%%%%%%%%%%%%%%%%%%%%%%%%%%%%%%%%%%%%%%%%%%%%%%%%%%%%%%%%%%%%%%%%%

%%%%%%%%%%%%%%%%%%%%%%%%%%%%%%%%%%%%%%%%%%%%%%%%%%%%%%%%%%%%%%%%%%%%%%%%%%%%%%%%%
\begin{remark}\label{R:mixed}
 This formalism extends to the case of mixed splines as studied in~\cite{DiP_Mixed,DiP_Assoc,GeSch98,Stiller83}.
 For each edge $\tau\in\Delta^0_1$ let $\alpha(\tau)$ be a nonnegative integer.
 Then $C^\alpha(\Delta)$ denotes the splines $(F_\sigma\mid\sigma\in\Delta_2)$ on $\Delta$ where if $\tau$ is
 an edge common to both $\sigma$ and $\sigma'$, then $G_\tau^{\alpha(\tau)+1}$ divides the difference
 $F_\sigma-F_{\sigma'}$.
 This is the kernel of the map of graded modules
\[
    \bigoplus_{\sigma\in\Delta_2} S \ \xrightarrow{\ \partial_2\ }\ 
    \bigoplus_{\tau\in\Delta^\circ_1} S/\langle G_\tau^{\alpha(\tau)+1}\rangle\,.
\]
 This formalism extends as well to splines over cell complexes $\Delta$ of any dimension
 whose cells are semialgebraic sets.
 We leave the corresponding statements to the reader.
\end{remark}
%%%%%%%%%%%%%%%%%%%%%%%%%%%%%%%%%%%%%%%%%%%%%%%%%%%%%%%%%%%%%%%%%%%%%%%%%%%%%%%%%

%%%%%%%%%%%%%%%%%%%%%%%%%%%%%%%%%%%%%%%%%%%%%%%%%%%%%%%%%%%%%%%%%%%%%%%%%%%%%%%%%
\section{Semialgebraic splines with a single vertex I}\label{S:pencil}
We consider the first nontrivial case of semialgebraic splines---when the complex $\Delta$ has a single
interior vertex $\upsilon$ and the forms defining the edges incident on $\upsilon$ form a pencil.
That is, they span a two-dimensional subspace in the space of all forms of degree $n$ vanishing at $\upsilon$.
This is always the case when the edges are line segments with at least two distinct slopes.
We determine the Hilbert polynomial of the spline module, showing that the
multiplicity of the scheme $S/J(\upsilon)$ is $n^2$ times the multiplicity of the scheme $S/I$, where $I$ is
an ideal generated by powers of linear forms vanishing at $\upsilon$.
This has a simple form, which we give in Corollary~\ref{C:multiplicity}.

This shows that the Hilbert polynomial of the spline module does not depend upon the real (as in real-number)
geometry of the curves underlying the edges $\tau$---it is independent of whether or not the curves are singular at
$\upsilon$ or at any other point, and whether or not the other points at which they meet are real, complex, or 
at infinity. 

Suppose that $L_1,\dotsc,L_s$ are linear forms in $\R[x,y]$ defining distinct lines through the origin, so that they
are pairwise coprime, and let \defcolor{$I$} be the ideal generated by the powers $L_1^{r+1},\dotsc,L_s^{r+1}$.
Observe that any $t\leq r{+}2$ of these powers are linearly independent ($r{+}2$ is the dimension of the space of
forms of degree $r{+}1$).
Recall that $S/I$ has a unique (up to change of basis) \demph{minimal free resolution} of the form
\[
     F_\bullet\ \colon\ 0\ \longrightarrow\ F_\delta
                  \ \xrightarrow{\ \psi_{\delta}\ }\ F_{\delta-1}
                  \ \xrightarrow{\ \psi_{\delta-1}\ }\ \dotsb\  \xrightarrow{\ \psi_1\ }\  S
\]
with $\mbox{coker }\psi_1=S/I$ and where the free module $F_i$ equals $\bigoplus_j S(-a_{ij})$.  
The index $\delta$ of the last nonzero free module is the projective dimension of $S/I$. 
By the Hilbert Syzygy theorem~\cite[Cor.~19.7]{Eisenbud}, the projective dimension of an ideal in a
polynomial ring is bounded above by the number of variables (in our case, three). 
The \demph{Castelnuovo-Mumford regularity} (henceforth \demph{regularity}) of $S/I$ is the number
$\max_{i,j}\{a_{ij}-i\}$.

We use the following results of Schenck and Stillman~\cite{ShSt97a}, describing the minimal free resolution
and regularity of an ideal of powers of linear forms in two variables. 

%%%%%%%%%%%%%%%%%%%%%%%%%%%%%%%%%%%%%%%%%%%%%%%%%%%%%%%%%%%%%%%%%%%%%%%%%%%%%%%%%
\begin{proposition}[\cite{ShSt97a}, Thm.~3.1]\label{P:ShSt}
 Let $I=\langle L_1^{r+1},\dotsc,L_t^{r+1}\rangle$ be an ideal minimally generated by the given powers of
 linear forms $L_1,\dotsc,L_t\in R:=\R[x,y]$ with $t>1$.
 A minimal free resolution of $R/I$ is given by
 \begin{equation}\label{Eq:LMFR}
   R(-r{-}1{-}a)^{s_1}\oplus R(-r{-}2{-}a)^{s_2} \ \longrightarrow\ R(-r{-}1)^t\ \longrightarrow\ R\,,
 \end{equation}
 where we have $s_1:=(t{-}1)a{+}t{-}r{-}2$ and $s_2:=r{+}1{-}(t{-}1)a$ with 
 $\defcolor{a}:=\lfloor\frac{r+1}{t-1}\rfloor\geq 1$.
\end{proposition}
%%%%%%%%%%%%%%%%%%%%%%%%%%%%%%%%%%%%%%%%%%%%%%%%%%%%%%%%%%%%%%%%%%%%%%%%%%%%%%%%%

If $m$ is the remainder of $r{+}1$ divided by $t{-}1$, then 
$s_1=t{-}1{-}m>0$  and $s_2=m\geq 0$.
Hence $I$ always has syzygies of degree $r{+}1{+}a$.

%%%%%%%%%%%%%%%%%%%%%%%%%%%%%%%%%%%%%%%%%%%%%%%%%%%%%%%%%%%%%%%%%%%%%%%%%%%%%%%%%
\begin{corollary}[\cite{ShSt97a}, Cor.~3.4]\label{C:PowerRegularity}
 The regularity of $R/I$ is $r{+}\lceil\frac{r{+}1}{t{-}1}\rceil-1$.
\end{corollary}
%%%%%%%%%%%%%%%%%%%%%%%%%%%%%%%%%%%%%%%%%%%%%%%%%%%%%%%%%%%%%%%%%%%%%%%%%%%%%%%%%

%%%%%%%%%%%%%%%%%%%%%%%%%%%%%%%%%%%%%%%%%%%%%%%%%%%%%%%%%%%%%%%%%%%%%%%%%%%%%%%%%
\begin{remark}\label{R:all_monoms}
 As $R/I$ is a finite-length module, the highest degree of a nonzero
 element in $R/I$ equals the regularity of $R/I$~\cite[Cor.~4.4]{EisenbudSyz}.  
 It follows from Corollary~\ref{C:PowerRegularity} that $I$ contains all monomials of degree at least 
 $r{+}\lceil\frac{r{+}1}{t{-}1}\rceil$, and thus the ideal $IS\subset S$ contains all monomials of $S$ where the
 degree in $x,y$ is at least $r{+}\lceil\frac{r{+}1}{t{-}1}\rceil$. 
\end{remark}
%%%%%%%%%%%%%%%%%%%%%%%%%%%%%%%%%%%%%%%%%%%%%%%%%%%%%%%%%%%%%%%%%%%%%%%%%%%%%%%%%

Tensoring the minimal free resolution~\eqref{Eq:LMFR} of $R/I$ with $S$ gives a minimal free
resolution of $IS$ (as $S$ is a flat $R$-module).
Taking Euler-Poincar\'e characteristic gives a formula for the multiplicity
of the scheme defined by $IS$, which is the Hilbert polynomial of $S/IS$,
\[
   s_1\tbinom{d-(r+1+a)+2}{2}+s_2\tbinom{d-(r+2+a)+2}{2}\ -\ 
   t\tbinom{d-(r+1)+2}{2}\ +\ \tbinom{d+2}{2}\ .
\]
This simplifies nicely.

%%%%%%%%%%%%%%%%%%%%%%%%%%%%%%%%%%%%%%%%%%%%%%%%%%%%%%%%%%%%%%%%%%%%%%%%%%%%%%%%%
\begin{corollary}\label{C:multiplicity}
 The multiplicity of the scheme defined by $I$ is $\binom{a{+}r{+}2}{2}-t\binom{a{+}1}{2}$. 
\end{corollary}
%%%%%%%%%%%%%%%%%%%%%%%%%%%%%%%%%%%%%%%%%%%%%%%%%%%%%%%%%%%%%%%%%%%%%%%%%%%%%%%%%

%%%%%%%%%%%%%%%%%%%%%%%%%%%%%%%%%%%%%%%%%%%%%%%%%%%%%%%%%%%%%%%%%%%%%%%%%%%%%%%%%
\subsection{Curves in a pencil}\label{SS:pencil}
Now we suppose that $G_1,\ldots,G_N$ are forms of degree $n$ that underlie the edges of $\Delta$, all of
which are incident on the point $\upsilon=[0:0:1]$.
Suppose that these forms define \defcolor{$s$} distinct algebraic curves, that $G_1$ and $G_2$ are relatively
prime, and each form $G_i$ lies in the linear span of $G_1$ and $G_2$, so the curves lie in a pencil.

%%%%%%%%%%%%%%%%%%%%%%%%%%%%%%%%%%%%%%%%%%%%%%%%%%%%%%%%%%%%%%%%%%%%%%%%%%%%%%%%%
\begin{proposition}\label{P:Pencil}
 Set $\defcolor{t}:=\min\{s,r{+}2\}$, and suppose that $G_1,\ldots,G_t$ define distinct curves.
 Then the ideal $\defcolor{J}:=\langle G_i^{r+1}\mid i=1,\dotsc,N\rangle$ is minimally generated by
 $G_1^{r+1},\ldots,G_t^{r+1}$.   
 Set $\defcolor{a}:=\lfloor\frac{r+1}{t-1}\rfloor$.  
 A minimal free resolution of $S/J$ is given by
 \begin{equation}\label{Eq:MMFR}
     S((-r{-}1{-}a)n)^{s_1}\oplus S((-r{-}2{-}a)n)^{s_2} \ \longrightarrow\ S((-r{-}1)n)^t\ \longrightarrow\ S\,,
 \end{equation}
 where $s_1:=(t{-}1)a{+}t{-}r{-}2$ and $s_2:=r{+}1{-}(t{-}1)a$.
\end{proposition}
%%%%%%%%%%%%%%%%%%%%%%%%%%%%%%%%%%%%%%%%%%%%%%%%%%%%%%%%%%%%%%%%%%%%%%%%%%%%%%%%%

%%%%%%%%%%%%%%%%%%%%%%%%%%%%%%%%%%%%%%%%%%%%%%%%%%%%%%%%%%%%%%%%%%%%%%%%%%%%%%%%%
\begin{proof}
 Since the forms $G_1$ and $G_2$ are relatively prime, they form a regular sequence.
% ($G_1\neq 0$ and $G_2$ is not a zero divisor in $S/\langle G_1\rangle$).
 In this situation, Hartshorne~\cite{Ha} showed that the map $\varphi\colon\defcolor{T}:=\R[u_1,u_2]\to S$ defined
 by $u_i\mapsto G_i$ for $i=1,2$ is an injection, and that $S$ is flat as a $T$-module.

 Let $L_1,\dotsc,L_N$ be the linear forms in $T$ such that $\varphi(L_i)=G_i$ for $i=1,\dotsc,N$
 and let $I:=\varphi^{-1}(J)$, which is the ideal $\langle L_1^{r+1},\dotsc,L_N^{r+1}\rangle$.
 As $s$ of the $G_i$ define distinct curves, the corresponding $s$ linear forms are pairwise relatively prime, and
 their powers generate $I$. 
 Then $t=\min\{s,r{+}2\}$ of these powers are linearly independent and thus are  minimal generators of the
 ideal $I$. 
 As $G_1,\dotsc,G_t$ are distinct, the powers $L_1^{r+1},\dotsc,L_t^{r+1}$ minimally generate $I$.
 Since $\varphi$ is injective and $\varphi(T)=\R[G_1,G_2]$ contains the generators of $J$, we conclude that
 $J$ is minimally generated by $G_1^{r+1},\dotsc,G_t^{r+1}$.

 Applying $\varphi$ to the exact sequence~\eqref{Eq:LMFR} %of free $T$-modules 
 and extending scalars to $S$
 gives the sequence~\eqref{Eq:MMFR} of free $S$-modules.
% Observe that t
 The degrees change, as %the map 
 $\varphi$ is a map of graded rings only if 
 $\deg(u_i)=\deg(G_i)=n$.
 The sequence remains exact, as $S$ is flat over $\varphi(T)$, and so it is a resolution of $S/J$.
 It remains minimal, as no map has a component of degree zero.
\end{proof}
%%%%%%%%%%%%%%%%%%%%%%%%%%%%%%%%%%%%%%%%%%%%%%%%%%%%%%%%%%%%%%%%%%%%%%%%%%%%%%%%%%

%%%%%%%%%%%%%%%%%%%%%%%%%%%%%%%%%%%%%%%%%%%%%%%%%%%%%%%%%%%%%%%%%%%%%%%%%%%%%%%%%%
\begin{corollary}\label{C:Pencil_multiplicity}
Let $J,n,N,a,t,s_1,s_2$ be as in Proposition~\ref{P:Pencil}.  The spline module $C^r(\Delta)$ is free as an $S$-module.  More precisely,
\[
C^r(\Delta)\simeq S\oplus S(-(r+1)n)^{N-t}\oplus S(-(r+1+a)n)^{s_1} \oplus S(-(r+2+a)n)^{s_2}.
\]
Its Hilbert function is
\[
\tbinom{d+2}{2}+(N{-}t)\tbinom{d-(r+1)n+2}{2}+s_1\tbinom{d-(r+1+a)n+2}{2}+s_2\tbinom{d-(r+2+a)n+2}{2}.
\]
The multiplicity of the scheme defined by $J$ equals
$n^2\bigl(\binom{a{+}r{+}2}{2}-t\binom{a{+}1}{2}\bigr)$. 
The Hilbert polynomial for the spline module is
\[
N\tbinom{d-(r+1)n+2}{2}\ +\ n^2\bigl(\tbinom{a{+}r{+}2}{2}-t\tbinom{a{+}1}{2}\bigr)\,,
\]
where we consider these binomial coefficients as polynomials in $d$.  
The postulation number is $(r+1+\lceil \frac{r+1}{t-1}\rceil)n-3$.
\end{corollary}
%%%%%%%%%%%%%%%%%%%%%%%%%%%%%%%%%%%%%%%%%%%%%%%%%%%%%%%%%%%%%%%%%%%%%%%%%%%%%%%%%%

%%%%%%%%%%%%%%%%%%%%%%%%%%%%%%%%%%%%%%%%%%%%%%%%%%%%%%%%%%%%%%%%%%%%%%%%%%%%%%%%%%
\begin{proof}
 By Proposition~\ref{P:exact_sequences}, $C^r(\Delta)\simeq S\oplus H_1(\calJ)$ and $H_1(\calJ)$ is the module of
 syzygies on  $\{G_1^{r+1},\ldots,G_N^{r+1}\}$.  
 Let these be ordered so that $\{G_1^{r+1},\ldots,G_t^{r+1}\}$ minimally generate $J$, while
 each $G_{t+i}^{r+1}$ for $i=1,\dotsc,N{-}t$ is a linear combination of $\{G_1^{r+1},\ldots,G_t^{r+1}\}$.
%  $\{G_{t+1}^{r+1},\ldots,G_N^{r+1}\}$ can be written as polynomial combinations of the minimal generators.  
 Then
\[
   H_1(\calJ)\ \simeq\ S(-(r{+}1)n)^{N-t}\oplus\mbox{syz}(J(\upsilon))\,,
\]
with a copy of $S(-(r{+}1)n)$ encoding the expression of $G_{t+i}$ in terms of the minimal generators of $J$.  
The module $\mbox{syz}(J(\upsilon))$ is the leftmost module in the minimal free resolution of $S/J(\upsilon)$ given in
Proposition~\ref{P:Pencil}.  
It is free because $J(\upsilon)$ has projective dimension two.  
The structure of $C^r(\Delta)$ as a free $S$-module follows.  
We deduce the Hilbert function and polynomial from this.  The postulation number $d_0$ is the largest integer which is less than at least one of the roots of the polynomials appearing as numerators in the binomial coefficients in the expression defining the Hilbert function, hence
\[
d_0=\left\lbrace
\begin{array}{ll}
(r+1+a)n-3 & \mbox{if } s_2=0\\
(r+2+a)n-3 & \mbox{otherwise}
\end{array},
\right.
\]
which is the same as $(r+1+\lceil \frac{r+1}{t-1}\rceil)n-3$.
\end{proof}
%%%%%%%%%%%%%%%%%%%%%%%%%%%%%%%%%%%%%%%%%%%%%%%%%%%%%%%%%%%%%%%%%%%%%%%%%%%%%%%%%%

Observe that the multiplicity of the scheme defined by $J$ is the product of the multiplicity, $n^2$
of the scheme defined by $\langle G_1,\dotsc,G_N\rangle=\langle G_1,G_2\rangle$ and the multiplicity of the
scheme defined by powers of linear forms as in Corollary~\ref{C:multiplicity}.

%%%%%%%%%%%%%%%%%%%%%%%%%%%%%%%%%%%%%%%%%%%%%%%%%%%%%%%%%%%%%%%%%%%%%%%%%%%%%%%%%
\begin{remark}\label{R:no_geometry}
 The Hilbert function of the spline module $C^r(\Delta)$ when the forms
 underlying the edges lie in a pencil depends only on the numerical invariants $N,s,r,n$ and {\it not} on the
 geometry in $\R^2$ of  the curves underlying the edges.
 We illustrate this remark by considering several cases when $n=2$ so that the edges are conics that lie in a
 pencil. 

 Let $G_1,G_2,G_3\in\R[x,y,z]$ be nonproportional quadratic forms with $G_3\in J:=\langle G_1,G_2\rangle$, so that
 the three lie in a pencil, and suppose also that they vanish at %the point 
 $\upsilon=[0:0:1]$.
 Then $J$ defines a zero-dimensional subscheme of $\C\P^2$ of multiplicity four.
 The $G_i$ are real, so there are several possibilities for the scheme defined by $J$ 
 %as  we also consider the real affine (
 in $\R^2$ (where $z\neq 0$).
 Figure~\ref{F:pencils} shows four cell complexes $\Delta$ with $|\Delta|$ the unit disc having
 three faces and three edges defined by the quadratic forms $G_1,G_2$, and $G_3$ in $\R^2$.
%the affine $\R^2$ defined by $z\neq 0$.
 Their spline modules all have the same Hilbert function and polynomial, which is displayed in
 Table~\ref{T:pencils}. 
%%%%%%%%%%%%%%%%%%%%%%%%%%%%%%%%%%%%%%%%%%%%%%%%%%%%%%%%%%%%%%%%%%%%%%%%%%%%%%%%%
\begin{table}[htb]
 \caption{Hilbert function and polynomial of three conics in a pencil.}
 \label{T:pencils}
  \begin{tabular}{|r|rrrrrrrrrrrrrr|c|}
    \hline  
    $r\backslash d$\rule{0pt}{12pt}
       &0&1&2&3&4&  5&  6&  7&  8&   9&  10&  11&  12&  13&Polynomial\\\hline
    0&1&3&7&13&22&34&49&67&88&112&139&169&202&238&\raisebox{-6pt}{\rule{0pt}{18pt}}%
                 $\frac{3}{2}d^2-\frac{3}{2}d+4$\\
    1&1&3&6&10&15&21&30&42&57& 75& 96&120&147&177&\raisebox{-6pt}{\rule{0pt}{18pt}}%
                 $\frac{3}{2}d^2-\frac{15}{2}d+21$\\
    2&1&3&6&10&15&21&28&36&46& 58& 73& 91&112&136&\raisebox{-6pt}{\rule{0pt}{18pt}}%
                 $\frac{3}{2}d^2-\frac{27}{2}d+58$\\
    3&1&3&6&10&15&21&28&36&45& 55& 66& 78& 93&111&\raisebox{-6pt}{\rule{0pt}{18pt}}%
                 $\frac{3}{2}d^2-\frac{39}{2}d+111$\\
    4&1&3&6&10&15&21&28&36&45& 55& 66& 78& 91&105&\raisebox{-6pt}{\rule{0pt}{18pt}}%
                 $\frac{3}{2}d^2-\frac{51}{2}d+184$\\ 
%    5&1&3&6&10&15&21&28&36&45& 55& 66& 78& 91&105&\raisebox{-6pt}{\rule{0pt}{18pt}}%
%                 $\frac{3}{2}d^2-\frac{63}{2}d+273$\\
%    6&1&3&6&10&15&21&28&36&45& 55& 66& 78& 91&105&\raisebox{-6pt}{\rule{0pt}{18pt}}%
%                 $\frac{3}{2}d^2-\frac{75}{2}d+382$\\
\hline
 \end{tabular}
%     
%
%0&1&3&7&13&22&34&49&67&88&112&139&169&202 &238&277&319&364&412&463&(3/2)*d^2-(3/2)*d+4\\
%1&1&3&6&10&15&21&30&42&57& 75& 96&120&147&177&210&246&285&327&372&(3/2)*d^2-(15/2)*d+21\\
%2&1&3&6&10&15&21&28&36&46& 58& 73& 91&112&136&163&193&226&262&301&(3/2)*d^2-(27/2)*d+58\\
%3&1&3&6&10&15&21&28&36&45& 55& 66& 78& 93&111&132&156&183&213&246&(3/2)*d^2-(39/2)*d+111\\
%4&1&3&6&10&15&21&28&36&45& 55& 66& 78& 91&105&121&139&160&184&211&(3/2)*d^2-(51/2)*d+184\\
%5&1&3&6&10&15&21&28&36&45& 55& 66& 78& 91&105&120&136&153&171&192&(3/2)*d^2-(63/2)*d+273\\
%6&1&3&6&10&15&21&28&36&45& 55& 66& 78& 91&105&120&136&153&171&190&(3/2)*d^2-(75/2)*d+382\\
%7&1&3&6&10&15&21&28&36&45& 55& 66& 78& 91&105&120&136&153&171&190&(3/2)*d^2-(87/2)*d+507\\
\end{table}
%%%%%%%%%%%%%%%%%%%%%%%%%%%%%%%%%%%%%%%%%%%%%%%%%%%%%%%%%%%%%%%%%%%%%%%%%%%%%%%%%

 Starting from the upper left and moving clockwise in Figure~\ref{F:pencils}, 
 we first have $G_1=x^2-6xy+y^2-2xz+6yz$, $G_2=x^2+6xy+y^2-2xz-6yz$, and $G_3=5G_1+4G_2$.
 These vanish at the four real points $[0:0:1]$, $[0:2:1]$, and $[1:\pm 1:1]$.
 Next, we have $G_1=x^2+xy+y^2-2xz$, $G_2=x^2+xy-2xz+2yz$, and $G_3=3G_1+2G_2$.
 These vanish at the two real points $[0:0:1]$, $[2:0:1]$, and the two complex points
 $[2:2\sqrt{-1}:1],[2:-2\sqrt{-1}:1]$. 
%%%%%%%%%%%%%%%%%%%%%%%%%%%%%%%%%%%%%%%%%%%%%%%%%%%%%%%%%%%%%%%%%%%%%%%%%%%%%%%%%
\begin{figure}[htb]

  \begin{picture}(190,130)
    \put(0,0){\includegraphics{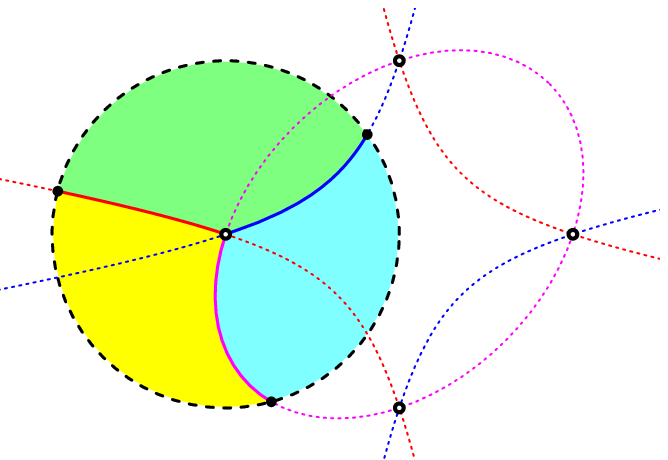}}
    \put(91,69){$G_1$} \put(34,78){$G_2$} \put(48,36){$G_3$} 
%    \put(154,47){\White{\rule{24pt}{14pt}}}\put(152,51){$[2:0:1]$}
%    \put(120, 10){\White{\rule{26pt}{12pt}}}\put(120, 12){$[1:-1:1]$}
%    \put(120,111){\White{\rule{26pt}{12pt}}}\put(120,113){$[1:1:1]$}
  \end{picture}
\qquad
  \begin{picture}(167,130)(0,-10)    %picture 112 high
    \put(0,0){\includegraphics{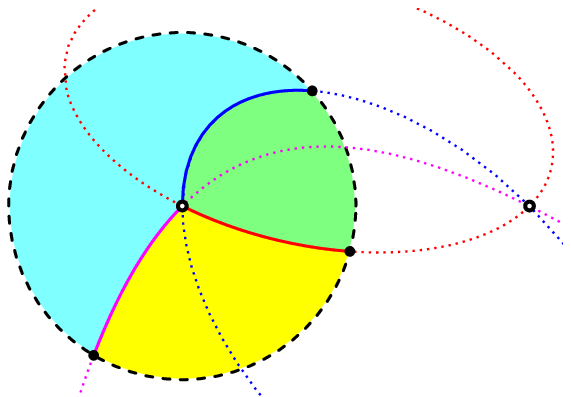}}
    \put(47,81){$G_1$} \put(23,38){$G_2$} \put(69,37){$G_3$} 
%    \put(142,40){\White{\rule{26pt}{12pt}}}\put(142,42){$[2:0:1]$}
  \end{picture}\newline
  \begin{picture}(172,112)%(0,-1)    %picture 112 high
    \put(0,0){\includegraphics{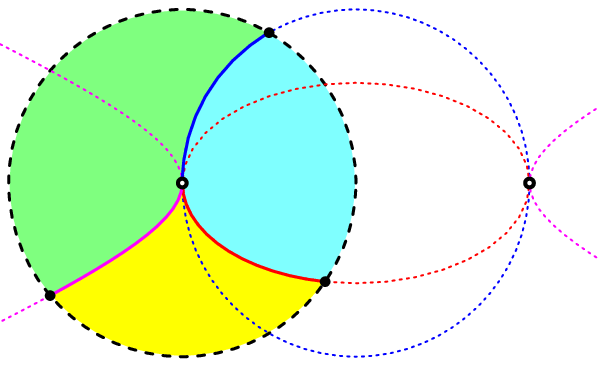}}
    \put(46,83){$G_1$} \put(69,37){$G_2$}  \put(20,37){$G_3$}
%    \put(124,52){$[2:0:1]$}
  \end{picture}
\qquad
  \begin{picture}(172,122)
    \put(0,0){\includegraphics{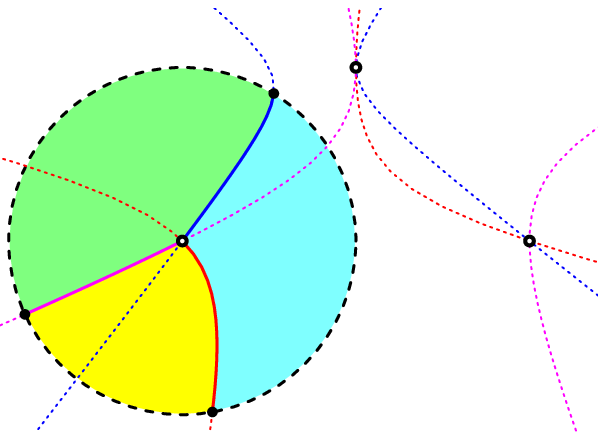}} % picture 122pt high
    \put(54,78){$G_1$} \put(65,29){$G_2$} \put(18,48){$G_3$} 
%    \put(126,46){$[2:0:1]$}
%    \put(105,102){$[1:1:1]$}
  \end{picture}

 \caption{(clockwise from upper left) 
      Four real points, two real and two complex points, two
      real and one double point, and two real double points.} 
 \label{F:pencils}
\end{figure}
%%%%%%%%%%%%%%%%%%%%%%%%%%%%%%%%%%%%%%%%%%%%%%%%%%%%%%%%%%%%%%%%%%%%%%%%%%%%%%%%%
For the third, let $G_1=2x^2+xy-2y^2-4xz+3yz$, $G_2=x^2+4xy-y^2-2xz-2yz$, and $G_3=6G_1-5G_2$.
These vanish at the points $[0:0:1]$, $[2:0:1]$, and $[1:1:1]$, sharing a common vertical tangent at the third
point, which has multiplicity two.
For our last pencil, let $G_1=x^2+y^2-2xz$, $G_2=x^2+3y^2-2xz$, and $G_3=4G_1-3G_2$.
These vanish only at the points $[0:0:1]$ and $[2:0:1]$, and share common vertical tangents at those points,
each of which has multiplicity two.
\end{remark}
%%%%%%%%%%%%%%%%%%%%%%%%%%%%%%%%%%%%%%%%%%%%%%%%%%%%%%%%%%%%%%%%%%%%%%%%%%%%%%%%%

%%%%%%%%%%%%%%%%%%%%%%%%%%%%%%%%%%%%%%%%%%%%%%%%%%%%%%%%%%%%%%%%%%%%%%%%%%%%%%%%%
\section{Semialgebraic splines with a single vertex II}\label{S:singleVertex}
Suppose that the complex $\Delta$ has a single interior vertex $\upsilon$,
but the forms defining the edges incident on $\upsilon$ are far from lying in a pencil in that
they have no other common zeroes in $\P^2(\C)$. 
We further suppose that the edge forms
are smooth at $\upsilon$ with distinct tangent directions.
Under these assumptions, we determine the Hilbert polynomial of the spline module by showing that the
multiplicities of the schemes $S/J(\upsilon)$ and $S/I$ are equal, where $I$ is generated by powers of the forms
defining the tangents at $\upsilon$.

Suppose that $\upsilon=[0:0:1]\in\P^2(\R)$ and there are $N$ interior edges incident on
$\upsilon$, defined by forms $G_1,\dotsc,G_N$, of degrees $n_1,\ldots,n_N$ with $[0:0:1]$ their only
common zero.
Expand each form $G_i$ as a polynomial in $z$,
\[
   G_i\ =\ \sum_{k=1}^{n_i} z^{n_i-k} G_{i,k}\,,
\]
where $G_{i,k}\in \R[x,y]$ has degree $k$.
Write $\defcolor{L_i}:=G_{i,1}$ for the coefficient of $z^{n_i-1}$ in $G_i$, which is
nonzero as $G_i$ is smooth at $\upsilon$. 
For an integer $r\geq 0$, let $\defcolor{J}:=J(\upsilon)$ be the ideal generated by
$G_1^{r+1},\dotsc,G_N^{r+1}$ and \defcolor{$I$} be the ideal generated by %the powers of the linear forms,
$L_1^{r+1},\dotsc,L_N^{r+1}$. 

%%%%%%%%%%%%%%%%%%%%%%%%%%%%%%%%%%%%%%%%%%%%%%%%%%%%%%%%%%%%%%%%%%%%%%%%%%%%%%%%%
\begin{theorem}\label{Th:generalPowers}
 When $L_1,\dotsc,L_N$ are distinct, the schemes $S/J$ and $S/I$ have the same
 Hilbert polynomial and degree.
\end{theorem}
%%%%%%%%%%%%%%%%%%%%%%%%%%%%%%%%%%%%%%%%%%%%%%%%%%%%%%%%%%%%%%%%%%%%%%%%%%%%%%%%%

We prove this in two steps.
In Subsection~\ref{SS:low} we show that when $r$ is small, these schemes coincide.
In Subsection~\ref{SS:high} we use toric degenerations to show that when $r$ is large, the
Hilbert polynomials are equal.

%%%%%%%%%%%%%%%%%%%%%%%%%%%%%%%%%%%%%%%%%%%%%%%%%%%%%%%%%%%%%%%%%%%%%%%%%%%%%%%%%
\begin{corollary}\label{Cor:generalPowers}
 The Hilbert polynomial of the spline module $C^r(\Delta)$ is
\[
  \sum_{i=1}^N\tbinom{d-(r+1)n_i+2}{2}\ +\ \tbinom{r+a+2}{2}-t\tbinom{a+1}{2}\,,
\]
 where $t:=\min\{N,r+2\}$ and $a:=\lfloor\frac{r+1}{t-1}\rfloor$.
\end{corollary}
\begin{proof}
This follows directly from Corollary~\ref{C:DimFormula}.
\end{proof}
%%%%%%%%%%%%%%%%%%%%%%%%%%%%%%%%%%%%%%%%%%%%%%%%%%%%%%%%%%%%%%%%%%%%%%%%%%%%%%%%%

%%%%%%%%%%%%%%%%%%%%%%%%%%%%%%%%%%%%%%%%%%%%%%%%%%%%%%%%%%%%%%%%%%%%%%%%%%%%%%%%%
\subsection{Low powers}\label{SS:low}

Let $G_1,\dotsc,G_N$, $L_1,\dotsc,L_N$, $I$, and $J$ be as above.
Suppose that $I$ is minimally generated by $t$ of the powers $L_i^{r+1}$.
We show that $S/J$ and $S/I$ define the same scheme when $2t\geq r+3$.
Let $\defcolor{\frakm}:=\langle x,y,z\rangle$ be the irrelevant ideal.
Recall that the saturation $(J:\frakm^\infty)$ of the ideal $J$ at $\frakm$ is 
$\{f\mid \exists k\mbox{ with }\frakm^k f\subset J\}$.
This defines the same projective scheme as does $J$.

%%%%%%%%%%%%%%%%%%%%%%%%%%%%%%%%%%%%%%%%%%%%%%%%%%%%%%%%%%%%%%%%%%%%%%%%%%%%%%%%%
\begin{lemma}\label{L:saturation}
 If $2t\geq r{+}3$, then $I = (J\colon \frakm^\infty)$.
\end{lemma}
%%%%%%%%%%%%%%%%%%%%%%%%%%%%%%%%%%%%%%%%%%%%%%%%%%%%%%%%%%%%%%%%%%%%%%%%%%%%%%%%%

%%%%%%%%%%%%%%%%%%%%%%%%%%%%%%%%%%%%%%%%%%%%%%%%%%%%%%%%%%%%%%%%%%%%%%%%%%%%%%%%%
\begin{corollary}\label{C:low_power}
 If $2t\geq r{+}3$, then $S/I$ and $S/J$ define the same scheme.
\end{corollary}
%%%%%%%%%%%%%%%%%%%%%%%%%%%%%%%%%%%%%%%%%%%%%%%%%%%%%%%%%%%%%%%%%%%%%%%%%%%%%%%%%

%%%%%%%%%%%%%%%%%%%%%%%%%%%%%%%%%%%%%%%%%%%%%%%%%%%%%%%%%%%%%%%%%%%%%%%%%%%%%%%%%
\begin{proof}[Proof of Lemma~\ref{L:saturation}]
 We first show that $J\subset I$.  Recall that $\deg(G_i)=n_i$.
 If we expand the form $G_i^{r+1}$ of degree $n_i(r{+}1)$ as a polynomial in decreasing powers of $z$, we obtain  
\[
   G_i^{r+1}\ =\ z^{(n_i-1)(r+1)}L_i^{r+1}\ +\ 
        \sum_{k=1}^{(n_i-1)(r+1)} z^{(n_i-1)(r+1)-k}K_{i,r+1+k}\,,
\]
 where $K_{i,a}\in\R[x,y]$ is homogeneous of degree $a$.
 Since this degree is at least $r{+}2$ and our hypothesis implies that 
 $2\geq (r{+}1)/(t{-}1)$, this degree is at least $r{+}\lceil\frac{r{+}1}{t{-}1}\rceil$.
 By Remark~\ref{R:all_monoms}, these polynomials $K_{i,a}$ lie in $I$.
 Since $L_i^{r+1}\in I$, we have that $G_i^{r+1}\in I$, and so $J\subset I$.

 As $\upsilon$ is the only zero of $J$, to show that 
 $I = (J\colon \frakm^\infty)$, we only need to show that the localizations at $\upsilon$ of $S/I$ and $S/J$
 are equal. 
 We assume that the forms have been ordered so that $L_1^{r+1},\dotsc,L_t^{r+1}$ are minimal generators
 of $I$.
 Let $J'\subset J$ be the ideal generated by $G_1^{r+1},\dotsc,G_t^{r+1}$.
 It suffices to show that $S/J'$ and $S/I$  have the same localization at $\upsilon$.
 Since $J'\subset J\subset I$, there are forms $A_{i,j}\in S=\R[x,y,z]$ such that 
 \begin{equation}\label{Eq:GintermsofL}
    G_i^{r+1}\ =\ \sum_{j=1}^t A_{i,j} L_j^{r+1}\,,
 \end{equation}
 for each $i=1,\dotsc,t$.
 To show that the localizations agree, we show that the matrix $A$ is invertible in 
 the localization $S_{\langle x,y\rangle}$ of $S$ at $\upsilon$, as 
 $\langle x,y\rangle$ defines the point $\upsilon$.
 
 Each form $A_{i,j}$ has degree $(n_i{-}1)(r{+}1)$.
 Let $A_{i,j}^{(n_i{-}1)(r{+}1)-k}$ denote the coefficient of $z^k$ in the expansion of $A_{i,j}$ as a polynomial in $z$.
 The highest power of $z$ appearing in~\eqref{Eq:GintermsofL} is $(n_i{-}1)(r{+}1)$.
 If we equate the coefficients of $z^{(n_i{-}1)(r{+}1)}$ in~\eqref{Eq:GintermsofL} (recalling that
 $G_i=z^{n_i-1}L_i+\dotsb$), we obtain
\[
    L_i^{r+1}\ =\ \sum_{j=1}^t A_{i,j}^0 L_j^{r+1}\,.
\]
 As these powers $L_1^{r+1},\dotsc,L_t^{r+1}$ are linearly independent, the matrix $A^0_{i,j}$ is the identity.

 In particular, the entries of the matrix $A$ that have a pure power of $z$ are exactly the diagonal entries.
 Thus its determinant has the form $z^{(-t+\sum_i n_i)(r+1)}+g$, where $g\in\langle x,y\rangle$, which implies that $A$
 is invertible in the local ring $S_{\langle x,y\rangle}$.
\end{proof}
%%%%%%%%%%%%%%%%%%%%%%%%%%%%%%%%%%%%%%%%%%%%%%%%%%%%%%%%%%%%%%%%%%%%%%%%%%%%%%%%%

\begin{remark}\label{R:Stiller}
Lemma~\ref{L:saturation} indicates how our results are complementary to Stiller's results in~\cite{Stiller83}.
His most general results in~\cite[\S~4]{Stiller83} require that the minimal generators of
$J(\upsilon)$, which have degree $n_i(r{+}1)$, are also minimal generators of the saturation of $J(\upsilon)$ (denoted $\scrI_X$ in~\cite{Stiller83}).
This assumption is also made in~\cite{Simis12}.  
Our assumptions that the edge forms are smooth at
$\upsilon$ and that $J(\upsilon)$ is supported only at $\upsilon$ will imply that, to the contrary, the saturation
of $J(\upsilon)$ is generated in degrees close to $r{+}1$.  
\end{remark}

%%%%%%%%%%%%%%%%%%%%%%%%%%%%%%%%%%%%%%%%%%%%%%%%%%%%%%%%%%%%%%%%%%%%%%%%%%%%%%%%%
\begin{remark}\label{R:rBig}
 The complex $\Delta$ in the left below has edges defined by the three homogeneous quadrics
 on the right.
 \[
   \raisebox{-55pt}{\begin{picture}(138,118)
     \put(0,0){\includegraphics{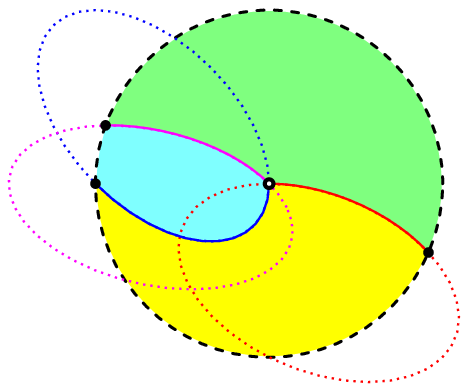}}
     \put(37,110){$G_1$} \put(129,35){$G_2$} \put(4,33){$G_3$} 
     \end{picture}}
  \qquad
   \raisebox{-2pt}{\begin{minipage}[b]{200pt}
    \[ \begin{array}{rcl}
      G_1&=& xz \ +\ x^2+xy+y^2\\\rule{0pt}{14pt}
      G_2&=& 2yz \ +\ x^2+xy+2y^2\\\rule{0pt}{14pt}
      G_3&=& \tfrac{3}{2}(x+y)z\ +\ x^2+xy+3y^2
     \end{array}\]
    \end{minipage}}
  \]
 Here $L_1=x$, $L_2=y$, and $L_3=x+y$ and the hypotheses of Lemma~\ref{L:saturation} hold for $r\leq 3$, hence
\[
   \bigl(\langle G_1^{r+1},G_2^{r+1},G_3^{r+1}\rangle\; :\; \frakm^\infty\bigr)
    \ =\ \langle L_1^{r+1},L_2^{r+1},L_3^{r+1}\rangle\,,
\]
 for $r=0,\dotsc,3$.
 However, for $r=4$, we do not have the containment
\[
   \langle G_1^5,G_2^5,G_3^5\rangle \ \subset\ \langle L_1^5,L_2^5,L_3^5\rangle\,.
\]
 If we set $J:=\langle G_1^5,G_2^5,G_3^5\rangle$, then 
\[
  (J\; :\; \frakm^\infty)\ =\ 
  \langle 5x^4y+10x^3y^2+10x^2y^3+5xy^4-y^5,\, x^5-y^5,\, y^6,\, xy^5,\, \defcolor{5x^2y^4+y^5z} \rangle\,.
\]
Each generator of $(J:\frakm^\infty)$ is in $\langle L_1^5,L_2^5,L_3^5\rangle$ except for the polynomial $5x^2y^4+y^5z$.
\end{remark}
%%%%%%%%%%%%%%%%%%%%%%%%%%%%%%%%%%%%%%%%%%%%%%%%%%%%%%%%%%%%%%%%%%%%%%%%%%%%%%%%%

%%%%%%%%%%%%%%%%%%%%%%%%%%%%%%%%%%%%%%%%%%%%%%%%%%%%%%%%%%%%%%%%%%%%%%%%%%%%%%%%%
\subsection{Distinct tangents}\label{SS:high}

The results of Subsection~\ref{SS:low} imply Theorem~\ref{Th:generalPowers} when $r$ is small
relative to $N$.
By Remark~\ref{R:rBig}, we cannot have $(J:\frakm^\infty)=I$ in general, so other arguments are needed.
We use toric degenerations to show that the schemes $S/J$ and $S/I$ have the same Hilbert polynomial.
We start with the following simple lemma.

%%%%%%%%%%%%%%%%%%%%%%%%%%%%%%%%%%%%%%%%%%%%%%%%%%%%%%%%%%%%%%%%%%%%%%%%%%%%%%%%%
\begin{lemma}\label{L:simple}
 Suppose that $I$ is an ideal of $S$ defining a scheme supported at $[0:0:1]$.
 Then $(I:\frakm^\infty)=(I:z^\infty)$.
\end{lemma}
%%%%%%%%%%%%%%%%%%%%%%%%%%%%%%%%%%%%%%%%%%%%%%%%%%%%%%%%%%%%%%%%%%%%%%%%%%%%%%%%%

%%%%%%%%%%%%%%%%%%%%%%%%%%%%%%%%%%%%%%%%%%%%%%%%%%%%%%%%%%%%%%%%%%%%%%%%%%%%%%%%%
\begin{proof}
 We always have $(I:\frakm^\infty)\subset(I:z^\infty)$.
 Since the only zero of $I$ is $[0:0:1]$, there is a $k>0$ such that $\langle x,y\rangle^k\subset I$.
 Let $f\in(I:z^\infty)$ so that there is some $\ell>0$ with $fz^\ell\in I$.
 Let $x^ay^bz^c$ be a monomial of degree at least $k+\ell$.
 Then either $a+b\geq k$ or $c\geq\ell$.
 In either case, $fx^ay^bz^c\in I$, and so $f\in(I:\frakm^k)\subset (I:\frakm^\infty)$, which completes the
 proof. 
\end{proof}
%%%%%%%%%%%%%%%%%%%%%%%%%%%%%%%%%%%%%%%%%%%%%%%%%%%%%%%%%%%%%%%%%%%%%%%%%%%%%%%%%

We recall the notion of initial degeneration of an ideal, which is explained in~\cite[\S~15]{Eisenbud}.
Given an integer vector $\defcolor{\omega}=(\omega_1,\omega_2,\omega_3)\in\Z^3$, a monomial $x^ay^bz^c$ of $S$
has a weight $a\omega_1+b\omega_2+c\omega_3$.
For a polynomial $F\in S$, let \defcolor{$\ini_\omega F$} be the sum of terms of $F$ whose monomials have
the largest weight with respect to $\omega$ among all terms of $F$.
The initial ideal of an ideal $I$ of $S$ with respect to $\omega$ is
\[
   \defcolor{\ini_\omega I}\ :=\ \langle \ini_\omega F \mid F\in I \rangle\,.
\]
The utility of this definition is that there is a flat degeneration of the scheme defined by $I$ into the
scheme defined by $\ini_\omega I$.
Consequently, $S/I$ and $S/\ini_\omega I$ have the same Hilbert function.
This flat degeneration is induced by the torus action on $S$ where
$\tau\in\C^\times$ acts on a monomial $x^ay^bz^c$ by $\tau^{-(a\omega_1+b\omega_2+c\omega_3)}x^ay^bz^c$, and
is called a toric degeneration.

We now fix the weight vector $\omega:=(0,0,1)$, so that $\ini_\omega F$ consists of the terms of $F$ 
with the highest power of $z$.
Let $G_1,\dotsc,G_N$ be forms in $S$ whose only common zero is $[0:0:1]$, so that 
the radical of  $J:=\langle G_1,\dotsc,G_N\rangle$ is $\langle x,y\rangle$.
For each $i$, let $c_i$ be the highest power of $z$ occurring in $G_i$ and define $\defcolor{F_i}\in\R[x,y]$ to
be the coefficient of $z^{c_i}$ in $G_i$, so that $z^{c_i}F_i=\ini_\omega G_i$.
Set $I:=\langle F_1,\dotsc,F_N\rangle$.

%%%%%%%%%%%%%%%%%%%%%%%%%%%%%%%%%%%%%%%%%%%%%%%%%%%%%%%%%%%%%%%%%%%%%%%%%%%%%%%%%
\begin{lemma}\label{L:initial_containment}
 If\/ $\ini_\omega J\subset I$, then $S/J$ and $S/I$ have the same Hilbert polynomial.
\end{lemma}
%%%%%%%%%%%%%%%%%%%%%%%%%%%%%%%%%%%%%%%%%%%%%%%%%%%%%%%%%%%%%%%%%%%%%%%%%%%%%%%%%

%%%%%%%%%%%%%%%%%%%%%%%%%%%%%%%%%%%%%%%%%%%%%%%%%%%%%%%%%%%%%%%%%%%%%%%%%%%%%%%%%
\begin{proof}
 We first observe that $\langle x,y\rangle$ is the radical of $\ini_\omega J$.
 Since $\sqrt{J}=\langle x,y\rangle$, there is some $k$ with 
 $\langle x,y\rangle^k\subset J$.
 As $\langle x,y\rangle^k$ is a monomial ideal, we have that 
 $\ini_\omega\langle x,y\rangle^k=\langle x,y\rangle^k$
 and hence $\langle x,y\rangle^k\subset \ini_\omega J$, which shows that 
 $\sqrt{\ini_\omega J}=\langle x,y\rangle$.

 Since $\ini_\omega G_i=z^{c_1}F_i$, we have $I\subset (\ini_\omega J\colon z^\infty)$.
 By Lemma~\ref{L:simple}, $( \ini_\omega J : z^\infty)=(\ini_\omega J\colon \frakm^\infty)$, so we have that 
 $I\subset (\ini_\omega J\colon \frakm^\infty)$.

 As $I$ is defined by polynomials in $x$ and $y$, if $z^cf\in I$, then $f\in I$, so that $I$ is
 saturated with respect to $\frakm=\langle x,y,z\rangle$.
 Saturating the inclusion $\ini_\omega J\subset I$ gives $(\ini_\omega J\colon \frakm^\infty)\subset I$.  Thus $(\ini_\omega J\colon \frakm^\infty)= I$ and $S/I$ and $S/\ini_\omega J$ have the same Hilbert polynomial.
 The lemma follows as $S/\ini_\omega J$ and $S/J$ have the same Hilbert polynomial, by flatness.
\end{proof}
%%%%%%%%%%%%%%%%%%%%%%%%%%%%%%%%%%%%%%%%%%%%%%%%%%%%%%%%%%%%%%%%%%%%%%%%%%%%%%%%%

We need to have that $\ini_\omega J\subset I$ to apply Lemma~\ref{L:initial_containment}.
By construction, the initial forms $\ini_\omega G_i=z^{c_i}F_i$  of the generators of $J$ lie in $I$.
To show that $\ini_\omega J\subset I$, we must understand what happens when there is cancellation of these
initial forms, which may be accomplished by understanding the syzygies of $I$

To that end, suppose that $F_1,\dotsc,F_N$ are minimal generators for $I$.
Write $a_{1i}$ for the degree of $F_i$.  
An ideal $I\subset S$ is \demph{Cohen-Macaulay} if the codimension of $S/I$ is equal to its projective
dimension.  
The ring $\R[x,y]/I$ has finite length (since $\sqrt{I}=\langle x,y \rangle$), so $I$ has projective dimension
two, which does not change if $I$ is considered as an ideal of $S$.  
Since $I$ has codimension two, it is Cohen-Macaulay and a structure theorem due to
Hilbert-Burch~\cite[Thm.~20.15]{Eisenbud} says that $S/I$ has a minimal free resolution of the form, 
 \begin{equation}\label{Eq:H-B}
   \bigoplus_{i=1}^{N-1}  S(-a_{2j})\ \longrightarrow\ 
   \bigoplus_{i=1}^{N}  S(-a_{1j})\ \longrightarrow\ S\,.
 \end{equation}
%

%%%%%%%%%%%%%%%%%%%%%%%%%%%%%%%%%%%%%%%%%%%%%%%%%%%%%%%%%%%%%%%%%%%%%%%%%%%%%%%%%
\begin{lemma}\label{L:syzygies}
 Given $I,J$ and $F_1,\dotsc,F_N$ as above,
 if $\max_{s,t}|a_{2s}-a_{2t}|\leq 2$, then $\ini_\omega J\subset I$.
\end{lemma}
%%%%%%%%%%%%%%%%%%%%%%%%%%%%%%%%%%%%%%%%%%%%%%%%%%%%%%%%%%%%%%%%%%%%%%%%%%%%%%%%%

%%%%%%%%%%%%%%%%%%%%%%%%%%%%%%%%%%%%%%%%%%%%%%%%%%%%%%%%%%%%%%%%%%%%%%%%%%%%%%%%%
\begin{example}
% Before giving a proof, we observe that t
 The condition on the second syzygies is necessary.
% Consider the ideal 
 Indeed, suppose that 
\[
   J\ :=\ \langle\, y^6+x^5z\,,\, 2x^2y^4+x^4yz\,,\, x^6 + y^5z \rangle\,,
\]
 so $I=\langle x^5, x^4y, y^5\rangle$ and $\sqrt{J}=\langle x,y\rangle$.
 The minimal free resolution of $S/I$ has the form
\[
    S(-6)\oplus S(-9) \ \longrightarrow\ 
    S(-5)^3\ \longrightarrow\ S\ \longrightarrow\ S/I\,,
\]
 so the condition on the second syzygies of Lemma~\ref{L:syzygies} does not hold. 
  Notice that
\[
     2x^3y^4{-}y^7\ =\ x(2x^2y^4{+}x^4yz){-}y(y^6{+}x^5z)
\]
is in the ideal $\ini_\omega J$ but not in the ideal $I$.  Using Macaulay2~\cite{M2}, we compute that the
multiplicity of $S/J$ is 20, while that of the scheme $S/I$ is 21. 
\end{example}
%%%%%%%%%%%%%%%%%%%%%%%%%%%%%%%%%%%%%%%%%%%%%%%%%%%%%%%%%%%%%%%%%%%%%%%%%%%%%%%%%

%%%%%%%%%%%%%%%%%%%%%%%%%%%%%%%%%%%%%%%%%%%%%%%%%%%%%%%%%%%%%%%%%%%%%%%%%%%%%%%%%
\begin{proof}[Proof of Lemma~\ref{L:syzygies}]
 Let $F\in J$ be a homogeneous form of degree $d$.
 We will show that $\ini_\omega F\in I$.
 Write $F$ in terms of the generators of $J$,
\[
   F\ =\ \sum_{i=1}^N H_i G_i\,,
\]
 where $H_1,\dotsc,H_N\in S$.

 Suppose that $\deg G_i=n_i$.
 Expanding $G_i$ as a polynomial in $z$ gives
\[
    G_i\ =\ \sum_{k=0}^{c_i} z^k g_{i,n_i-k}\,,
\]
 where $g_{i,n_i-k}\in\R[x,y]$ has degree $n_i{-}k$ and $c_i$ is the highest power of $z$ occurring in $G_i$.
 Note that $F_i=g_{i,n_i-c_i}$ and that $a_{1i}=n_i-c_i$ in the Hilbert-Burch resolution~\eqref{Eq:H-B}.

 Let $\gamma_i$ be the highest power of $z$ that occurs in $H_i$ and note that the degree \defcolor{$\eta_i$}
 of $H_i$ is $d{-}n_i$.
 Expand $H_i$ as a polynomial in $z$,
\[
   H_i\ =\ \sum_{k=0}^{\gamma_i} z^k h_{i,\eta_i-k}\,,
\]
 where $h_{i,\eta_i-k}\in\R[x,y]$ has degree $\eta_i-k$.

 If we expand $F$ as a polynomial in $z$, we have
\[
    F\ :=\ \sum_{k=0}^\mu  z^k f_{d-k}\,,
\]
 where  $f_{d-k}\in\R[x,y]$ has degree $d-k$ and $\mu$ is the maximum of $c_i+\gamma_i$.
 Then 
\[
  \ini_\omega F\ =\ z^m f_{d-m}\,,
   \qquad\mbox{ where } m\ :=\ \max\{k\mid f_{d-k}\neq 0\}\ \leq\ \mu\,.
\]

 Suppose that the forms are numbered so that for $i=1,\dotsc,p$ we have 
 $\mu=c_i+\gamma_i$, but if $i>p$, then $\mu>c_i+\gamma_i$.
 Then the coefficient $f_{d-\mu}$ of $z^\mu$ in $F$ is
 \begin{equation}\label{Eq:F-syzygy}
   f_{d-\mu}\ =\ \sum_{i=1}^p h_{i,\eta_i-\gamma_i} g_{i,n_i-c_i}\,.
 \end{equation}
 Since $g_{i,n_i-c_i}=F_i$ lies in $I$, if we have $f_{d-\mu}\neq 0$, then 
 $\ini_\omega F\in I$ as desired.

 Suppose on the contrary that $f_{d-\mu}= 0$.
 Since $I$ is minimally generated by $F_1,\dotsc,F_N$, the sum in~\eqref{Eq:F-syzygy} is a syzygy of the ideal
 $I$. 
 Then the degree $d{-}\mu$ of the sum~\eqref{Eq:F-syzygy} is at least one of the degrees $a_{2s}$ in the
 Hilbert-Burch resolution~\eqref{Eq:H-B}.
 Let $a^m_2$ be the minimum of the $a_{2s}$ and $a^M_2$ be the maximum.
 Then $a^m_2\leq d{-}\mu$, and so every term $f_{d-k}$ for $k<\mu$ in the expansion of $F$ with respect to
 $z$ has degree at least $a^m_2+1$ (recall that $f_{d-\mu}=0$).

 However, the regularity of $S/I$ is $a^M_2-2$, so that $I$ contains every monomial in $x,y$ of degree at least
 $a^M_2-1$. 
 Since $a^M_2\leq a^m_2+2$, we have $a^M_2-1\geq a^m_2+1$, so that every term  $f_{d-k}$ with
 $k<\mu$ in the expansion of $F$ lies in $I$, which implies that $F\in I$ and in particular $\ini_\omega F\in
 I$, which completes the proof.
\end{proof}
%%%%%%%%%%%%%%%%%%%%%%%%%%%%%%%%%%%%%%%%%%%%%%%%%%%%%%%%%%%%%%%%%%%%%%%%%%%%%%%%%

Let $G_1,\dotsc,G_N\in S$ be forms of the same degree $n$ with $[0:0:1]$ their only common zero
such that their linear terms $L_1,\dotsc,L_N$ at  $[0:0:1]$ are distinct and nonzero.

%%%%%%%%%%%%%%%%%%%%%%%%%%%%%%%%%%%%%%%%%%%%%%%%%%%%%%%%%%%%%%%%%%%%%%%%%%%%%%%%%
\begin{proof}[Proof of Theorem~\ref{Th:generalPowers}]
 For any $r\geq 0$, let $J$ be the ideal generated by $G_1^{r+1},\dotsc,G_N^{r+1}$ and let $I$ be the ideal
 generated by the powers $L_1^{r+1},\dotsc,L_N^{r+1}$ of linear forms.
 These powers of linear forms are distinct and they are linearly independent if and only if $N\leq r+2$. 

 Suppose first that $N>r+2$.  
 Then $I$ is generated by $t=r+2$ of these powers.
 In this case, $2t=2r+4>r+3$ and the theorem follows by Corollary~\ref{C:low_power}.

 If instead $N\leq r+2$, then $I$ is minimally generated by these powers.
 By Proposition~\ref{P:ShSt} the second syzygies of $I$ differ by at most one, so the hypotheses of
 Lemma~\ref{L:syzygies} hold and  $\ini_\omega J\subset I$.
 But then Lemma~\ref{L:initial_containment} implies the statement of the theorem.
\end{proof}
%%%%%%%%%%%%%%%%%%%%%%%%%%%%%%%%%%%%%%%%%%%%%%%%%%%%%%%%%%%%%%%%%%%%%%%%%%%%%%%%%

\begin{remark}\label{R:Mixed}
 Extensions of Theorem~\ref{Th:generalPowers} and Corollary~\ref{Cor:generalPowers} to
 mixed smoothness (where varying orders of continuity are imposed across interior edges) require
 a minimal free resolution for an ideal generated by arbitrary powers of linear forms in two variables.
 This is provided by Geramita and Schenck in~\cite{GeSch98}.  
 Note that the condition on second syzygies needed in Lemma~\ref{L:syzygies} is satisfied for
 ideals generated by arbitrary powers of linear forms.  
 We leave the details as an exercise for the interested reader.
\end{remark}

\begin{remark}
 When the hypotheses of Lemma~\ref{L:syzygies} hold, we may determine the multiplicity of $S/J(\upsilon)$.
 We cannot relax the condition on distinct tangents (see Example~\ref{ex:CLOex}), however the forms
 may have controlled singularities at $\upsilon$.  
 For instance, if each form has at worst a cusp singularity at $\upsilon$, then the ideal $I$
 defining the tangent cone is generated by (possibly different) powers of linear forms, similar to 
 Remark~\ref{R:Mixed}.  
 As long as the underlying linear forms are distinct, Lemma~\ref{L:syzygies} computes the multiplicity of
 $S/J(\upsilon)$.  
\end{remark}

%%%%%%%%%%%%%%%%%%%%%%%%%%%%%%%%%%%%%%%%%%%%%%%%%%%%%%%%%%%%%%%%%%%%%%%%%%%%%%%%%
\section{Hilbert Function and Regularity}\label{S:Regularity}

Suppose that the cell complex $\Delta$ has a single interior vertex, $\upsilon$, and that the forms defining
its edges are smooth at $\upsilon$ with distinct tangents as in Section~\ref{S:singleVertex}.
The formula of Corollary~\ref{Cor:generalPowers} for the Hilbert polynomial of the spline module
$C^r(\Delta)$ only gives the dimension of $C^r_d(\Delta)$ when $d$ exceeds the postulation
number of the spline module.
By Corollary~\ref{C:DimFormula}~\eqref{Eq:DimensionFormula}, the dimension of $C^r_d(\Delta)$
differs from an explicit polynomial by the dimension of $(S/J(\upsilon))_d$, which is the Hilbert function of
$S/J(\upsilon)$.  
%Whereas in Sections~\ref{S:pencil} and~\ref{S:singleVertex} we studied the long term behavior of $\dim
%(S/J(\upsilon))_d$, here w
We study the entire Hilbert function in some cases and give bounds on the postulation
number of $S/J(\upsilon)$ using (Castelnuovo-Mumford) regularity.

%%%%%%%%%%%%%%%%%%%%%%%%%%%%%%%%%%%%%%%%%%%%%%%%%%%%%%%%%%%%%%%%%%%%%%%%%%%%%%%%%
\subsection{Hilbert Function}
Stiller used Max Noether's `$AF+BG$ Theorem' to compute the Hilbert function of $S/J(\upsilon)$ in a special
case~\cite[Thm.~4.9]{Stiller83} (see Example~\ref{ex:multipoint}).
As explained by Eisenbud, Green, and Harris,~\cite{CayleyBacharach}, one generalization of Max Noether's theorem
leads to linkage.  
We use linkage to study the Hilbert function of $S/J(\upsilon)$ in some cases.  
We will only consider the case when $\Delta$ has three edges defined by pairwise
coprime forms $G_1,G_2,G_3$ of degrees $n_1,n_2,n_3$ as in Figure~\ref{F:cell_complex}.  
Set $J(\upsilon)=\langle G_1^{r+1},G_2^{r+1}, G_3^{r+1}\rangle$, $K:=\langle G_1^{r+1},G_2^{r+1}\rangle$,
$K':=(K:G_3^{r+1})$, and assume $G_3^{r+1}\notin K$ so that $K'\neq S$. 

%%%%%%%%%%%%%%%%%%%%%%%%%%%%%%%%%%%%%%%%%%%%%%%%%%%%%%%%%%%%%%%%%%%%%%%%%%%%%%%%%
\begin{lemma}
   The ideal $K'$ is Cohen-Macaulay of codimension two.
\end{lemma}
%%%%%%%%%%%%%%%%%%%%%%%%%%%%%%%%%%%%%%%%%%%%%%%%%%%%%%%%%%%%%%%%%%%%%%%%%%%%%%%%%

%%%%%%%%%%%%%%%%%%%%%%%%%%%%%%%%%%%%%%%%%%%%%%%%%%%%%%%%%%%%%%%%%%%%%%%%%%%%%%%%%
\begin{proof}
 Suppose that $P$ is an associated prime of $K'$ so that $P=\ann_{S/K'}(F)$ for some $F\notin K'$.  
 Then $P$ is an associated prime of $K$ since $P=\ann_{S/K}(FG_3^{r+1})$ and $FG_3^{r+1}\notin K$ since $F\notin
 K'=(K:G_3^{r+1})$.  
 Since $K$ is a complete intersection generated by two polynomials, it is Cohen-Macaulay of codimension two.  
 Hence all of its associated primes, including $P$,  have codimension two.  
 Then $\codim(K')=\codim(K)=2$.  
 Now it suffices to show that $\depth(S/K')=1$; in other words there exists an element of $S/K'$ that is not a zero divisor.
 Since $K$ is codimension two Cohen-Macaulay, $\depth(S/K)=1$ and there is some $H\in S$ which is not a zero divisor
 in $S/K$.  
 We do not have $H\in K'$ as $HG^{r+1}_3\notin K$.  
 Hence $H\notin K$.  
 Furthermore $H$ cannot be a zero divisor on $S/K'$---otherwise there exists $F\in S$ with $FH\in K'$, hence
 $H(FG^{r+1}_3)\in K$, contrary to the assumption that $H$ is not a zero divisor on $S/K$. 
 Hence $\depth(S/K')\ge 1$.  
 Since $\depth(S/K')\le\mbox{dim}(S/K')=1$, $\depth(S/K')= 1$, and so $K'$ is Cohen-Macaulay of codimension two, as
 claimed. 
\end{proof}

Set $K'':=(K:K')$.  
Since $K'$ is codimension two and Cohen-Macaulay, so is $K''$~\cite[Thm.~21.23]{Eisenbud}, and $(K:K'')=K'$.  
The ideals $K',K''$ are said to be \demph{linked}.  
There is a particularly nice relationship between the Hilbert functions of $K'$ and $K''$.

%%%%%%%%%%%%%%%%%%%%%%%%%%%%%%%%%%%%%%%%%%%%%%%%%%%%%%%%%%%%%%%%%%%%%%%%%%%%%%%%%
\begin{proposition}\cite[Thm.~CB7]{CayleyBacharach}\label{P:CB7}
  Let $K,K',K''$ be as above and set $s=n_1(r+1)+n_2(r+1)-3$.  Then
\[
\dim\left(K'/K\right)_d\ =\ \mult\left(S/K''\right)\ -\ \dim\left(S/K''\right)_{s-d}\,,
\]
 where $\dim(S/K'')_{s-d}$ is zero for $d>s$.
\end{proposition}
%%%%%%%%%%%%%%%%%%%%%%%%%%%%%%%%%%%%%%%%%%%%%%%%%%%%%%%%%%%%%%%%%%%%%%%%%%%%%%%%%

We show how this proposition may be used when $r=0$.

%%%%%%%%%%%%%%%%%%%%%%%%%%%%%%%%%%%%%%%%%%%%%%%%%%%%%%%%%%%%%%%%%%%%%%%%%%%%%%%%%
\begin{proposition}\label{P:RadHilbert}
 Suppose that $K=\langle G_1,G_2 \rangle$  defines $n_1n_2$ distinct points $\Gamma\subset\mathbb{P}^2(\C)$ and
 that on $\Gamma$, $G_3$ vanishes at only the point $\upsilon$.  
 Then 
\[
   \dim \left(S/J(\upsilon)\right)_d\ =\ \left\lbrace
   \begin{array}{lcr}
     \dim \left(S/K\right)_d-\dim \left(S/K\right)_{d-n_3} && d\le n_1+n_2+n_3-3\\
     1 && d\ge n_1+n_2+n_3-2
   \end{array}
   \right.\ .
\]
\end{proposition}
%%%%%%%%%%%%%%%%%%%%%%%%%%%%%%%%%%%%%%%%%%%%%%%%%%%%%%%%%%%%%%%%%%%%%%%%%%%%%%%%%

%%%%%%%%%%%%%%%%%%%%%%%%%%%%%%%%%%%%%%%%%%%%%%%%%%%%%%%%%%%%%%%%%%%%%%%%%%%%%%%%%
\begin{proof}
  As the points of $\Gamma$ are distinct, $K$ is is the (radical) ideal of all polynomials vanishing on
  $\Gamma$.  
  Since $G_3$ only vanishes at $\upsilon$, the ideal $K':=(K:G_3)$ is the ideal of
  $\Gamma\smallsetminus\{\upsilon\}$.
  Thus $K''=(K:K')$ is the ideal of $\upsilon$, so that $K''=\langle x,y \rangle$.  
  By Proposition~\ref{P:CB7},
 \[
    \dim(K'/K)_d\ =\ 1-\dim(S/K'')_{n_1+n_2-3-d}\ =\ 
     \left\lbrace
       \begin{array}{lcr}
        0 && d \le n_1+n_2-3\\
        1 && d > n_1+n_2-3
      \end{array}
     \right.\ .
 \]
 The ideal $K'$ is related to $J(\upsilon)=\langle G_1,G_2,G_3 \rangle$ via the multiplication sequence
 \begin{equation}\label{Eq:MultSequence}
   0\ \rightarrow\ S(-n_3)/K' \ \xrightarrow{\ \cdot G_3\ }\ 
    S/K\ \rightarrow\  S/J(\upsilon)\  \rightarrow\  0\,.
 \end{equation}
 Using~\eqref{Eq:MultSequence}, the tautological short exact sequence
\[
    0\ \rightarrow\  K'/K\ \rightarrow\ S/K\ \rightarrow\ S/K'\ \rightarrow\  0\,,
\]
  and taking Euler-Poincar\'e characteristic yields
\[
    \dim \left(S/J(\upsilon)\right)_d\ =\ \dim\left(S/K\right)_d\ -\ \dim\left(S/K\right)_{d-n_3}
                      \ +\ \dim\left(K'/K\right)_{d-n_3}\,.
\]
 Observing that $\dim(S/K)_d=n_1n_2$ for $d\ge n_1+n_2-2$ yields the result.
\end{proof}
%%%%%%%%%%%%%%%%%%%%%%%%%%%%%%%%%%%%%%%%%%%%%%%%%%%%%%%%%%%%%%%%%%%%%%%%%%%%%%%%%

%%%%%%%%%%%%%%%%%%%%%%%%%%%%%%%%%%%%%%%%%%%%%%%%%%%%%%%%%%%%%%%%%%%%%%%%%%%%%%%%%
\begin{remark}
The hypotheses of Proposition~\ref{P:RadHilbert} can be weakened to assume that $G_1,G_2$ define a complete intersection scheme $\Gamma\subset\mathbb{P}^2(\C)$ in which $\upsilon$ is a reduced point and $G_3$ vanishes only at $\upsilon\in\Gamma$.  This alteration does not change the conclusion.
\end{remark}

%%%%%%%%%%%%%%%%%%%%%%%%%%%%%%%%%%%%%%%%%%%%%%%%%%%%%%%%%%%%%%%%%%%%%%%%%%%%%%%%%
\begin{example}\label{E:Bacharach}
Suppose $n_1=n_2=n_3=3$.  
Then $K$ is a complete intersection generated by two polynomials of degree three.  
Hence we have
\[
   \begin{array}{|c|ccccr|}\hline
    d           & 0 & 1 & 2 & 3 & \geq 4\\\hline
    \dim(S/K)_d & 1 & 3 & 6 & 8 & 9\\\hline
   \end{array}
\]
By Proposition~\ref{P:RadHilbert},
\[
   \begin{array}{|c|ccccccr|}\hline
    d           & 0 & 1 & 2 & 3 & 4 & 5 & \geq 6\\\hline
    \dim(S/J(\upsilon))_d & 1 & 3 & 6 & 7 & 6 & 3 & 1\\\hline
   \end{array}
\]
\end{example}
%%%%%%%%%%%%%%%%%%%%%%%%%%%%%%%%%%%%%%%%%%%%%%%%%%%%%%%%%%%%%%%%%%%%%%%%%%%%%%%%%

%%%%%%%%%%%%%%%%%%%%%%%%%%%%%%%%%%%%%%%%%%%%%%%%%%%%%%%%%%%%%%%%%%%%%%%%%%%%%%%%%
\begin{remark}
 The computation $\dim(S/J(\upsilon))_6=1$ has a hidden application of the classical
 Cayley-Bacharach theorem; namely that $\dim(K'/K)_3=0$.  
 This statement says that any cubic vanishing on eight of the nine points defined by $K$ must also vanish on the
 ninth point.  
 Proposition~\ref{P:CB7} generalizes this classical result.
\end{remark}

%%%%%%%%%%%%%%%%%%%%%%%%%%%%%%%%%%%%%%%%%%%%%%%%%%%%%%%%%%%%%%%%%%%%%%%%%%%%%%%%%
\subsection{Regularity}
As determining the Hilbert function of $S/J(\upsilon)$ is quite difficult, we turn now to bounding its postulation
number.  This is controlled by the regularity of $S/J(\upsilon)$.

%%%%%%%%%%%%%%%%%%%%%%%%%%%%%%%%%%%%%%%%%%%%%%%%%%%%%%%%%%%%%%%%%%%%%%%%%%%%%%%%%
\begin{proposition}[\cite{EisenbudSyz}, Thm.~4.2]\label{P:RegularityAndPostulation}
 The Hilbert function $\dim (S/J(\upsilon))_d$ agrees with the Hilbert polynomial for 
 $d\geq \reg(S/J(\upsilon))+1$.  
 Thus, the postulation number of $S/J(\upsilon)$ is at most the regularity of $S/J(\upsilon)$.
\end{proposition}
%%%%%%%%%%%%%%%%%%%%%%%%%%%%%%%%%%%%%%%%%%%%%%%%%%%%%%%%%%%%%%%%%%%%%%%%%%%%%%%%%

The regularity of quotients  $S/I$ for some ideal $I$ has been studied intensively.  
One of the tightest general bounds applicable to our situation is due to
Chardin and Fall~\cite{Chardin05}. 

\begin{proposition}\cite[Cor.~0.2]{Chardin05}\label{P:ChardinReg}
Let $S$ be a polynomial ring in three variables and $I$ an ideal generated in degree at most $n$ satisfying
$\dim(S/I)\le 1$.  
Then $\reg(S/I)\le 3(n-1)$.
\end{proposition} 

\begin{corollary}\label{C:GeneralPost}
Suppose $\Delta$ has a single interior vertex $\upsilon$ and $N$ edges defined by forms $G_1,\ldots,G_N$ of degrees
$n_1\le\cdots\le n_N=n$, meeting smoothly at $\upsilon$ with distinct tangents.  
Set $t=\min\{N,r+2\}$ and $a=\lfloor\frac{r+1}{t-1}\rfloor$.  
Then
\[
\dim C^r_d(\Delta)\ =\ \sum\limits_{i=1}^N\tbinom{d-n_i(r+1)+2}{2}+\tbinom{r+a+2}{2}-t\tbinom{a+1}{2}
\]
for $d\ge 3n(r+1)-2$.
\end{corollary}
\begin{proof}
This follows from Corollary~\ref{Cor:generalPowers}, Proposition~\ref{P:RegularityAndPostulation}, and
Proposition~\ref{P:ChardinReg}. 
\end{proof}

The bound in Corollary~\ref{C:GeneralPost} is not optimal (see
Table~\ref{T:PostulationCompare}).  
We derive a tighter bound when $\Delta$ has three edges defined by forms
$G_1,G_2,G_3$ of degrees $n_1,n_2,n_3$ meeting at a single interior vertex $\upsilon$, as in
Figure~\ref{F:cell_complex}. 
We take $\upsilon$ to be the point $[0:0:1]$ with ideal $\langle x,y\rangle$.
Then $J(\upsilon):=\langle G_1^{r+1},G_2^{r+1},G_3^{r+1}\rangle$.
As we have seen, even in this simple case determining the Hilbert function is difficult.  
(See also~\cite{Simis12}, where three-generated ideals in $\C[x,y,z]$ are studied in the context of plane Cremona
maps.)  

Our regularity bound is a translation of~\cite[Thm.~1.2]{Simis12}.
We use \demph{local cohomology}.
Let $\frakm=\langle x,y,z \rangle$.
The zeroth local cohomology of an $S$-module $M$ is
\[
   \defcolor{H^0_{\frakm}(M)}\ :=\ \{m\in M \mid \frakm^k m=0\mbox{ for some $k\geq 0$}\}.
\]
If $I$ is an ideal of $S$, then $H^0_\frakm(S/I)=(I:\frakm^\infty)/I$.

For $i>0$, the local cohomology functors $H^i_\frakm(\ )$ for $i>0$ are the 
\demph{right derived functors} of $H^0_\frakm(\ )$.
If $M\rightarrow\calI$ is an injective resolution of $M$ then $\defcolor{H^i_\frakm(M)}:=H^i(H^0_\frakm(\calI))$,
the $i^{th}$ cohomology of the complex $H^0_\frakm(\calI)$. 
For more on local cohomology, see~\cite[App.~1]{EisenbudSyz}.
If $M$ is graded and finitely generated, then $H^i_{\frakm}(M)$ is
graded and \demph{Artinian} in that $H^i_\frakm(M)_d=0$ for $d\gg 0$. 
Also $H^i_\frakm(M)\neq 0$ only for the range 
$\depth(M)\le i\le \dim(M)$~\cite[Prop.~A1.16]{EisenbudSyz}.  
These are important since $\reg(M)$ may be identified using local cohomology.\medskip

%%%%%%%%%%%%%%%%%%%%%%%%%%%%%%%%%%%%%%%%%%%%%%%%%%%%%%%%%%%%%%%%%%%%%%%%%%%%%%%%%
\begin{proposition}{~\cite[Thm~4.3]{EisenbudSyz}}\label{P:LCohoReg}
If $M$ is a finitely generated graded $S$-module, then $\reg(M)$ is the smallest integer $d$ satisfying:
 \begin{enumerate}
 	\item $H^0_\frakm(M)_d\neq 0$, and 
 	\item $H^i_{\frakm}(M)_{i+d-1}=0$ for all $i>0$.\rule{0pt}{13pt}
 \end{enumerate}
\end{proposition}
%%%%%%%%%%%%%%%%%%%%%%%%%%%%%%%%%%%%%%%%%%%%%%%%%%%%%%%%%%%%%%%%%%%%%%%%%%%%%%%%%

Let \defcolor{$(S/I)^*$} be the graded dual of $S/I$, in degree $-d$ it is the dual vector space to
$(S/I)_d$.  
Given a graded $S$-module $M$ of finite length, \defcolor{$\indeg(M)$} is the lowest degree of a
nonzero homogeneous component of $M$ and \defcolor{$\nd(M)$} is the highest degree.
%Oftentimes $\nd(M)$ is called the socle degree of $M$.
% This is not needed.

\begin{proposition}\label{P:RegBound}
 Let $G_1,G_2,G_3$ be forms of degrees $1\le n_1\le n_2\le n_3$, with $n_3\ge 2$ whose only common zero
 in $\P^2$ is $\upsilon$.  
 Then
\[
   \reg(S/J(\upsilon))\ \leq (n_1{+}n_2{+}n_3{-}1)(r{+}1)-3.
   \]
If $2t\geq r{+}3$ and $n_1>1$, with $t$ as in the statement of Lemma~\ref{L:saturation}, then equality holds.
\end{proposition}

\begin{proof}
Let $J=J(\upsilon)$.  We show first that $\nd(H^0_\frakm(S/J))\le (n_1{+}n_2{+}n_3{-}1)(r{+}1)-3$, with
equality if $2t\geq r+3$.  This bound will follow from~\cite[Lem.~5.8]{Chardin04}.  

Specializing the second part of~\cite[Lem.~5.8]{Chardin04} to the case $i=0$ gives
\[
H^0_\frakm(S/J)((n_1{+}n_2{+}n_3)(r{+}1){-}3)\ \cong\  H^0_\frakm(S/J)^*\,,
\]
where the $(n_1{+}n_2{+}n_3)(r{+}1){-}3$ in parentheses denotes a graded shift of $H^0_\frakm(S/J)$. 

Using the identification $H^0_\frakm(S/J)=(J:\frakm^\infty)/J$, this yields
\[
   \bigl((J:\frakm^{\infty})/J\bigr)((n_1{+}n_2{+}n_3)(r{+}1){-}3)\ \cong\
   \left((J:\frakm^{\infty})/J\right)^*.
\]
This implies that 
\[
   \indeg((J:\frakm^{\infty})/J)\ -\ (n_1{+}n_2{+}n_3)(r{+}1){+}3
     \ =\ -\nd((J:\frakm^{\infty})/J)\,,
\]
so
\[
   \indeg((J:\frakm^{\infty})/J)\ +\ 
   \nd((J:\frakm^{\infty})/J)\ =\ (n_1{+}n_2{+}n_3)(r{+}1){-}3\,.
\]
 Compare this to the first statement of~\cite[Thm.~1.2]{Simis12}.  
 As $J$ is $\langle x,y \rangle$-primary, $(J:\frakm^\infty)=(J:z^\infty)$.
 Since no pure power of $z$ appears in any of the forms $G_1,G_2,G_3$ (they all vanish at $[0:0:1]$), the
 maximum power of $z$ in $G_i^{r+1}$ is $(n_i{-}1)(r{+}1)$.  
 Hence $\indeg((J:z^{\infty})/J)\ge r+1$, so
 \begin{eqnarray*}
   \nd\left((J:\frakm^{\infty})/J\right) & =&
        (n_1{+}n_2{+}n_3)(r{+}1){-}3\ -\ 
        \indeg\left(J:\frakm^{\infty}/J\right) \\ 
    &\leq& (n_1{+}n_2{+}n_3)(r{+}1){-}3\ -\ (r{+}1)\\
    &=&(n_1{+}n_2{+}n_3{-}1)(r{+}1){-}3\,.
 \end{eqnarray*}
Hence $\nd(H^0_\frakm(S/J))\le (n_1+n_2+n_3-1)(r+1)-3$, as desired.  
If $2t\geq r{+}3$, then Lemma~\ref{L:saturation} shows that $(J:\frakm^{\infty})=\langle
L_1^{r+1},L_2^{r+1},L_3^{r+1}\rangle$, hence $\indeg((J:\frakm^\infty)/J)=r+1$ and
$\nd(H^0_\frakm(S/J))=(n_1{+}n_2{+}n_3{-}1)(r{+}1){-}3$.  

Now we show that $\nd(H^1_\frakm(S/J))\le(n_1+n_2)(r+1)-4$.  
Compare this statement to the second part of~\cite[Thm.~1.2]{Simis12}.  
By local duality~\cite[Thm.~A1.9]{EisenbudSyz},
\[
    H^1_\frakm(S/J)\ \cong\ \Ext^2(S/J,S(-3))^*\,.
\]
Hence $\nd(H^1_\frakm(S/J))=-\indeg(\Ext^2(S/J,S(-3)))$.  
Now let $I=\langle G_1^{r+1},G_2^{r+1} \rangle\subset J$.  
Then $I$ is a complete intersection, so $S/I$ has a minimal free resolution of the form
\[
   0\ \longrightarrow\ S(-a-b)\ \longrightarrow\ S(-a)\oplus S(-b)\ 
    \longrightarrow\ S,
\]
where $a=\deg(G_1^{r+1})=n_1(r+1)$ and $b=\deg(G_2^{r+1})=n_2(r+1)$.  In particular,
$\Ext^2(S/I,S)\cong S(a+b)/I$.  Let $\gamma=G_3^{r+1}$ and set $c=\deg(\gamma)=n_3(r+1)$.  
We have a short exact sequence
 \[
  0\ \longrightarrow\ S(-c)/(I:\gamma)
   \ \xrightarrow{\ \cdot\gamma\ }\  S/I
   \  \longrightarrow\ S/J\  \longrightarrow\ 0\,.
\]
Since $\codim(S/(I:\gamma))\ge 2$, the long exact sequence in $\Ext$ yields
\[
  0\ \longrightarrow\ \Ext^2(S/J,S)\ \longrightarrow
   \ \Ext^2(S/I,S)\ \xrightarrow{\ \cdot\gamma\ }\ \Ext^2(S/(I:\gamma),S)\
   \longrightarrow\ \dotsb\,,
\]
hence $\Ext^2(S/J,S)$ is the kernel of the map given by multiplication by $\gamma$ on $\Ext^2(S/I,S)$.  
Since $\Ext^2(S/I,S)\cong S(a+b)/I$, we have
\[
     \Ext^2(S/J,S)\ \cong\ \bigl((I:\gamma)/I\bigr)(a+b)\,,
\]
where the $a+b$ in parentheses represents a degree shift.  
Since $\gamma\notin I$, it follows that $I:\gamma$ is generated in degrees $\ge 1$, hence
$\indeg(\Ext^2(S/J,S))\ge -a-b+1$. 
It follows that $\nd(H^1_\frakm(S/J))=-\indeg(\Ext^2(S/J,S(-3)))\le a+b-4$.

Now, by Proposition~\ref{P:LCohoReg},
 \begin{eqnarray*}
   \reg(S/I) & =& \max\{\nd(H^0_\frakm(S/I)),\nd(H^1_\frakm(S/I))+1\}\\ 
   & \le& \max\{(n_1+n_2+n_3-1)(r+1)-3,(n_1+n_2)(r+1)-3\}\\
   &=&(n_1+n_2+n_3-1)(r+1)-3\,,
 \end{eqnarray*}
as we assume $n_1,n_2,n_3$ are all at least one.  
Equality holds if $2t\ge r+3$.
\end{proof}

\begin{corollary}\label{C:ThreeCurvesPostulation}
If $\Delta$ has three edges defined by forms $G_1,G_2,G_3$ of degree $n_1,n_2,n_3$ meeting smoothly with
distinct tangents at a single interior vertex, 
then 
\[
    \dim C^r_d(\Delta)\ =\ \sum\limits_{i=1}^3\tbinom{d-n_i(r+1)+2}{2}+\tbinom{r+a+2}{2}-t\tbinom{a+1}{2}
\]
for $d\ge (n_1+n_2+n_3-1)(r+1)-2$, 
where $t=\min\{3,r+2\}$ and $a=\lfloor\frac{r+1}{t-1}\rfloor$.
\end{corollary}

\begin{proof}
 This follows from Corollary~\ref{Cor:generalPowers}, Proposition~\ref{P:RegBound}, and
 Proposition~\ref{P:RegularityAndPostulation}. 
\end{proof}

\begin{example}\label{ex:Compare}
For the cell complex of Figure~\ref{F:cell_complex}, whose Hilbert function, polynomial, and
postulation number are shown in Table~\ref{T:HilbertFunction}, we show in Table~\ref{T:PostulationCompare} how
the bounds on the postulation number in Corollary~\ref{C:GeneralPost} and
Corollary~\ref{C:ThreeCurvesPostulation} compare with the actual postulation number, $d_0$.
This indicates that for $r$ small, we should expect the bound in
Corollary~\ref{C:ThreeCurvesPostulation} to be close to exact, while the bound in
Corollary~\ref{C:GeneralPost} may be quite far off.

%%%%%%%%%%%%%%%%%%%%%%%%%%%%%%%%%%%%%%%%%%%%%%%%%%%%%%%%%%%%%%%%%%%%%%%%%%%%%%%%%
\begin{table}[htb]

 \caption{Comparing bounds with the postulation number in Example~\ref{ex:Compare}.}
 \label{T:PostulationCompare}

\begin{tabular}{|c|ccccccccccccccc|}
\hline
$r$ & 0 & 1 & 2 & 3 & 4 & 5 & 6 & 7 & 8 & 9&10&11&12&13&14\raisebox{-6pt}{\rule{0pt}{18pt}}\\
\hline
$d_0$ & 1 & 5 & 9 &13 &17 &20 &24 &28 &32 &35&39&43&47&50&54\raisebox{-6pt}{\rule{0pt}{18pt}}\\
Cor.~\ref{C:ThreeCurvesPostulation} & 1 & 5 & 9 & 13 & 17 & 21 & 25 & 29 &
        33 & 37&41&45&49&53&57\raisebox{-6pt}{\rule{0pt}{18pt}}\\
Cor.~\ref{C:GeneralPost} & 4 & 10 & 16 & 22 & 28 & 34 & 40 & 46 & 52 &
       58&64&70&76&82&88\raisebox{-6pt}{\rule{0pt}{18pt}}\\
\hline
\end{tabular}
\end{table}
%%%%%%%%%%%%%%%%%%%%%%%%%%%%%%%%%%%%%%%%%%%%%%%%%%%%%%%%%%%%%%%%%%%%%%%%%%%%%%%%%

\end{example}

%%%%%%%%%%%%%%%%%%%%%%%%%%%%%%%%%%%%%%%%%%%%%%%%%%%%%%%%%%%%%%%%%%%%%%%%%%%%%%%%%
\section{Examples}\label{S:examples}

We illustrate some limitations and possible extensions of our results for cell complexes
$\Delta$ with a single interior vertex $\upsilon$.
In Section~\ref{S:pencil}, we determined the Hilbert function of the spline module when the curves defining the
edges of $\Delta$ lie in a pencil.
As noted in Remark~\ref{R:no_geometry}, this Hilbert function does not depend upon the geometry of the curves
in that pencil, only on their number and degree.
In Section~\ref{S:singleVertex}, we determined the Hilbert polynomial of the spline module in nearly the opposite
case---when the curves vanish simultaneously only at the vertex $\upsilon$ and they have distinct
tangents at $\upsilon$.

Our first example is from~\cite{CLO05}---three curves of different degrees, but only two tangents
at $\upsilon$.
In the remaining examples, the curves are three conics defining schemes of multiplicity three and two
(intermediate between the cases of Sections~\ref{S:pencil} and~\ref{S:singleVertex}).
We show that the Hilbert polynomial of the spline module depends upon the geometry of the curves.

%%%%%%%%%%%%%%%%%%%%%%%%%%%%%%%%%%%%%%%%%%%%%%%%%%%%%%%%%%%%%%%%%%%%%%%%%%%%%%%%%
\begin{example}\label{ex:CLOex}
 This example appears in~\cite[\S~8.3, Exer.~13]{CLO05}. 
 Let $\Delta$ consist of portions of the three curves $G_1=yz-x^2$, $G_2=xz+y^2$, and $G_3=yz^2-x^3$ in the unit
 disc in $\R^2$ where $z\neq 0$ meeting at the origin as in Figure~\ref{fig:CLOex}.   
%%%%%%%%%%%%%%%%%%%%%%%%%%%%%%%%%%%%%%%%%%%%%%%%%%%%%%%%%%%%%%%%%%%%%%%%%%%%%%%%%
\begin{figure}[htb]
  \begin{picture}(103,103)
   \put(0,0){\includegraphics{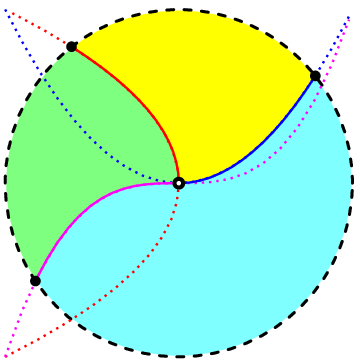}}
   \put(68,74){$G_1$} \put(40,79){$G_2$}\put(15,50){$G_3$}
  \end{picture}   
 \caption{Complex $\Delta$ in Example~\ref{ex:CLOex}.}
 \label{fig:CLOex}
\end{figure}
%%%%%%%%%%%%%%%%%%%%%%%%%%%%%%%%%%%%%%%%%%%%%%%%%%%%%%%%%%%%%%%%%%%%%%%%%%%%%%%%%
 We have $J(\upsilon)=\langle G_1^{r+1}, G_2^{r+1}, G_3^{r+1} \rangle$ and 
 $C^r(\Delta)\cong S\oplus \mbox{syz}(J(\upsilon))$ by Proposition~\ref{P:exact_sequences}. 
	
 The tangents of $G_1,G_2,G_3$ at $(0,0)$ are $L_1=y$, $L_2=x$, and $L_3=y$.  
 Let $I=\langle L_1^{r+1}, L_2^{r+1}, L_3^{r+1} \rangle=\langle x^{r+1},y^{r+1} \rangle$.  
 Since the tangents are not distinct, we cannot use Theorem~\ref{Th:generalPowers} to compute the multiplicity of the
 scheme $S/J(\upsilon)$.   
 However, if $r\le 1$, then the schemes $S/I$ and $S/J(\upsilon)$ have the same multiplicity by
 Corollary~\ref{C:low_power}.  
 Using Corollary~\ref{C:DimFormula},
 \[
    \HP(C^r(\Delta),d)\ =\ 2\tbinom{d-2(r+1)}{2}+\tbinom{d-3(r+1)}{2}+(r{+}1)^2\,,
 \]
 if $r\le 1$.  
 For $r\ge 2$, we replace $G_3^{r+1}$ by $G_3^{r+1}-z^{r+1}G_1^{r+1}$, which has leading term  $x^2y^rz^{2r+1}$ in $z$.
 Set $I':=\langle x^{r+1},x^2y^r,y^{r+1}\rangle$.
 The minimal free resolution of $S/I'$ has the form
 \begin{equation}\label{Eq:NMFR}
   0\ \longrightarrow\ S(-2r-1)\oplus S(-r-3)\ \longrightarrow\ S^3\ 
    \longrightarrow\ S\,.
 \end{equation}
The ideal $I'$ is generated by the leading forms of  $G_1^{r+1}$ $G_2^{r+1}$, and $G_3^{r+1}{-}z^{r+1}G_1^{r+1}$,
which generate $J(\upsilon)$.   
By Lemmas~\ref{L:initial_containment} and~\ref{L:syzygies}, $S/J(\upsilon)$ has the same Hilbert polynomial as
$S/I'$ when $2\le r\le 4$.  Using Corollary~\ref{C:DimFormula}, 
 \begin{eqnarray*}
   \HP(C^r(\Delta),d) &=&2\tbinom{d-2(r+1)}{2}+\tbinom{d-3(r+1)}{2}+\HP(S/J(\upsilon),d)\\
	&=&2\tbinom{d-2(r+1)}{2}+\tbinom{d-3(r+1)}{2}+\HP(S/I',d)\\
	&=&2\tbinom{d-2(r+1)}{2}+\tbinom{d-3(r+1)}{2}+(2+r+r^2)\,,
  \end{eqnarray*}
 where the final equality follows from the minimal free resolution of $I'$~\eqref{Eq:NMFR}.

 If $r>4$, the techniques of this paper will not suffice to compute $\HP(C^r(\Delta),d)$.  
 Computations in Macaulay2 show that the saturation of $\ini_\omega(J(\upsilon))$ is 
 $I'=\langle x^{r+1},x^2y^r,y^{r+1}\rangle$ for $r=5$, where $\omega=(0,0,1)$.  
 This cannot be concluded from Lemma~\ref{L:syzygies} since the condition on second syzygies fails.  
 Further computations in Macaulay2 show
\[
    \bigl(\ini_\omega(J(\upsilon)):\frakm^\infty\bigr)\ =\ 
    \langle x^{r+1},y^{r+1},x^2y^r,\defcolor{x^6y^{r-1}} \rangle
\]
 for $r=6,7,8,9$. 
 For $r=10$,
 \[
   \bigl(\ini_\omega(J(\upsilon)):\frakm^\infty\bigr)\ =\ 
     \langle x^{r+1},y^{r+1},x^2y^r,x^6y^{r-1},x^{10}y^{r-2} \rangle\,,
 \]
 indicating a growth in the number of generators of the saturation
 $(\ini_\omega(J(\upsilon)):\frakm^\infty)$.  See Table~\ref{T:CLOEx}.  
 For $r\geq 5$ different techniques will be needed to compute
 $\HP(C^r(\Delta),d)$.  The column headed $d_0$ gives the postulation number (computed only through $r=5$).
 The final column is the regularity bound from Corollary~\ref{C:ThreeCurvesPostulation}.

%%%%%%%%%%%%%%%%%%%%%%%%%%%%%%%%%%%%%%%%%%%%%%%%%%%%%%%%%%%%%%%%%%%%%%%%%%%%%%%%%
\begin{table}[htb]
 \caption{Table for cell complex $\Delta$ in Example~\ref{ex:CLOex}.}
 \label{T:CLOEx}

\begin{tabular}{|c|c|c|c|c|c|}
\hline\raisebox{-6pt}{\rule{0pt}{18pt}}
$r$ & $\mbox{sat}(\ini_\omega(J(\upsilon)))$ & $\HP(C^r(\Delta),d)$ & $\HP(S/J(\upsilon),d)$ & $d_0$ & $6(r{+}1){-}3$ \\
\hline\raisebox{-6pt}{\rule{0pt}{18pt}}
0 & $\langle x,y \rangle$ & $\frac{3}{2}d^2-\frac{5}{2}d+2$ & 1 & 3 & 3 \\ \raisebox{-6pt}{\rule{0pt}{18pt}}
1 & $\langle x^2,y^2\rangle$ & $\frac{3}{2}d^2-\frac{19}{2}d+20$ & 4 & 9 & 9 \\\raisebox{-6pt}{\rule{0pt}{18pt}}
2 & $\langle x^3,y^3,x^2y^2 \rangle$ & $\frac{3}{2}d^2-\frac{33}{2}d+56$ & 8 & 15 & 15\\\raisebox{-6pt}{\rule{0pt}{18pt}}
3 & $\langle x^4,y^4,x^2y^3 \rangle$ & $\frac{3}{2}d^2-\frac{47}{2}d+111$ & 14 & 21 & 21 \\\raisebox{-6pt}{\rule{0pt}{18pt}}
4 & $\langle x^5,y^5,x^2y^4 \rangle$ &  $\frac{3}{2}d^2-\frac{61}{2}d+185$ & 22 & 27 & 27 \\
\hline\raisebox{-6pt}{\rule{0pt}{18pt}} 
5 & $\langle x^6,y^6,x^2y^5 \rangle$ & $\frac{3}{2}d^2-\frac{75}{2}d+278$ & 32 & 32 & 33 \\\raisebox{-6pt}{\rule{0pt}{18pt}}
6 & $\langle x^7,y^7,x^2y^6,x^6y^5\rangle$ &
   $\frac{3}{2}d^2-\frac{89}{2}d+389$ & 43 & & 39 \\\raisebox{-6pt}{\rule{0pt}{18pt}}
7 & $\langle x^8,y^8,x^2y^7,x^6y^6\rangle$ &
   $\frac{3}{2}d^2-\frac{103}{2}d+519$ & 56 & & 45 \\\raisebox{-6pt}{\rule{0pt}{18pt}}
8 & $\langle x^9,y^9,x^2y^8,x^6y^7\rangle$ &
   $\frac{3}{2}d^2-\frac{117}{2}d+668$ & 71 & & 51 \\\raisebox{-6pt}{\rule{0pt}{18pt}}
9 & $\langle x^{10},y^{10},x^2y^9,x^6y^8\rangle$ &
    $\frac{3}{2}d^2-\frac{131}{2}d+836$ & 88 & & 57 \\\raisebox{-6pt}{\rule{0pt}{18pt}}
10& $\langle x^{11},y^{11},x^2y^{10},x^6y^9,x^{10}y^8\rangle$ & $\frac{3}{2}d^2-\frac{145}{2}d+1022$ & 106 & & 63 \\
\hline
\end{tabular}

\end{table}	
%%%%%%%%%%%%%%%%%%%%%%%%%%%%%%%%%%%%%%%%%%%%%%%%%%%%%%%%%%%%%%%%%%%%%%%%%%%%%%%%%
	
\end{example}
%%%%%%%%%%%%%%%%%%%%%%%%%%%%%%%%%%%%%%%%%%%%%%%%%%%%%%%%%%%%%%%%%%%%%%%%%%%%%%%%%
     
%%%%%%%%%%%%%%%%%%%%%%%%%%%%%%%%%%%%%%%%%%%%%%%%%%%%%%%%%%%%%%%%%%%%%%%%%%%%%%%%%

\begin{example}\label{ex:multipoint}
 Suppose that $G_1$, $G_2$, and $G_3$ are conics underlying the edges of a cell complex $\Delta$ with a single
 vertex $\upsilon$ that do not lie in a pencil, but simultaneously vanish in at least another point.
 By Corollary~\ref{C:DimFormula}~\eqref{Eq:DimensionFormula}, the Hilbert function of $C^r(\Delta)$ is
\[
    \HF(C^r(\Delta),d)\ =\ 
   3\tbinom{d{-}2r}{2} + \dim(S/J(\upsilon))_d\,.
\]
 We compute the Hilbert functions of $S/J(\upsilon)$ for different choices of three conics.
 Stiller~\cite[Thm.~4.9]{Stiller83} did this when $r=0$ and when the conics define 1, 2, or 3
 simple points.

 We first consider three cases where the conics define a scheme of multiplicity three, consisting of the three
 points $\upsilon=[0:0:1]$, $[2:0:1]$, and $[1:-1:1]$.
 The first triple is 
 $A:= 2x^2+2xy+y^2-4xz-3y$, 
 $B:=  x^2-xy+y^2-2xz+yz$, and 
 $C:=  x^2-8xy-y^2-2xz+6yz$.
 Their curves have distinct tangents at each of three points. 
 The next triple is
 $A$, 
 $D:= x^2+4xy-y^2-2xz-6yz$, and 
 $E:= x^2-3xy-y^2-2xz+yz$.
 The curves of $D$ and $E$ are tangent at $[1:-1:1]$.
 The third triple is 
 $D$, 
 $E$, and
 $F:= 2x^2+5xy+y^2-4xz-6yz$.
 The curves of $F$ and $D$ are also tangent at the point $[2:0:1]$.
 We display the resulting cell complexes in the affine $\R^2$ with $z\neq 0$ in Figure~\ref{F:Mult_th}.
%%%%%%%%%%%%%%%%%%%%%%%%%%%%%%%%%%%%%%%%%%%%%%%%%%%%%%%%%%%%%%%%%%%%%%%%%%%%%%%%%
\begin{figure}[htb]

   \begin{picture}(140,105)
     \put(0,0){\includegraphics{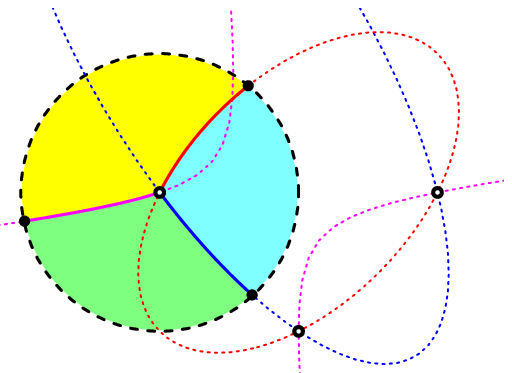}}
     \put(61,36){$A$} \put(49,70){$B$} \put(21,50){$C$}
   \end{picture}
   \quad
   \begin{picture}(140,105)
     \put(0,0){\includegraphics{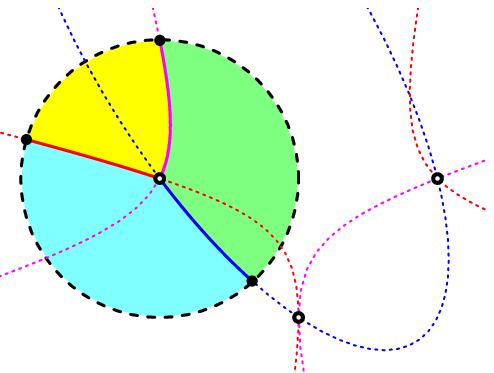}}
     \put(50,30){$A$} \put(20,51){$D$} \put(50,74){$E$}
   \end{picture}
   \quad
   \begin{picture}(140,105)
     \put(0,0){\includegraphics{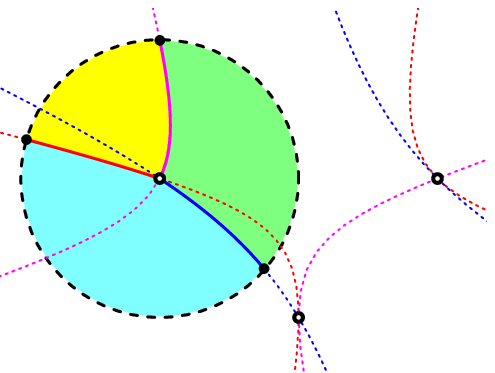}}
     \put(53,34){$F$} \put(20,51){$D$} \put(50,74){$E$}
   \end{picture}

\caption{Three conics defining three points.}
\label{F:Mult_th}
\end{figure}
%%%%%%%%%%%%%%%%%%%%%%%%%%%%%%%%%%%%%%%%%%%%%%%%%%%%%%%%%%%%%%%%%%%%%%%%%%%%%%%%%

Table~\ref{Ta:ThreeQuadrics} gives the Hilbert functions for $d\leq 18$ and $r\leq 4$ for each of these triples.
While the Hilbert functions agree for $r=0$ (as shown by Stiller~\cite[Thm.~4.9]{Stiller83}), they differ for
all larger $r$ in both the postulation number and Hilbert polynomial.

%%%%%%%%%%%%%%%%%%%%%%%%%%%%%%%%%%%%%%%%%%%%%%%%%%%%%%%%%%%%%%%%%%%%%%%%%%%%%%%%%
\begin{table}[htb]
 \caption{Hilbert functions for three quadrics defining three points.}
 \label{Ta:ThreeQuadrics}

\begin{center}
{\small
\begin{tabular}{|r|ccccccccccccccccccc|}\hline
&\multicolumn{19}{c|}{$d$}\\\hline
$r$&0&1&2&3&4&5&6&7&8&9&10&11&12&13&14&15&16&17&18\\\hline
%%%%%%%%%%%%%%%%%%%%%%%%%%%   Three simple
0&\defcolor{1}& 3& 3& 3& 3& 3& 3& 3& 3& 3& 3& 3& 3& 3& 3& 3& 3& 3& 3\\
1& 1& 3& 6&10&12&12&\defcolor{10}& 9& 9& 9& 9& 9& 9& 9& 9& 9& 9& 9& 9\\
2& 1& 3& 6&10&15&21&25&27&27&\defcolor{25}&21&21&21&21&21&21&21&21&21\\
3& 1& 3& 6&10&15&21&28&36&42&46&48&48&46&\defcolor{42}&36&36&36&36&36\\
4& 1& 3& 6&10&15&21&28&36&45&55&63&69&73&75&75&73&69&\defcolor{63}&57\\\hline
\end{tabular}\vspace{5pt}

\begin{tabular}{|r|ccccccccccccccccccc|}\hline
&\multicolumn{19}{c|}{$d$}\\\hline
$r$&0&1&2&3&4&5&6&7&8&9&10&11&12&13&14&15&16&17&18\\\hline
%%%%%%%%%%%%%%%%%%%%%%%%%%%   One tangent
0&\defcolor{1}&3&3& 3& 3& 3& 3& 3& 3& 3& 3& 3& 3& 3& 3& 3& 3& 3& 3\\
1& 1&3&6&10&12&\defcolor{12}&10&10&10&10&10&10&10&10&10&10&10&10&10\\
2& 1&3&6&10&15&21&25&27&27&\defcolor{25}&22&22&22&22&22&22&22&22&22\\
3& 1&3&6&10&15&21&28&36&42&46&48&48&46&\defcolor{42}&38&38&38&38&38\\
4& 1&3&6&10&15&21&28&36&45&55&63&69&73&75&75&73&69&\defcolor{63}&60\\\hline 
\end{tabular}\vspace{5pt}

\begin{tabular}{|r|ccccccccccccccccccc|}\hline
&\multicolumn{19}{c|}{$d$}\\\hline
$r$&0&1&2&3&4&5&6&7&8&9&10&11&12&13&14&15&16&17&18\\\hline
%%%%%%%%%%%%%%%%%%%%%%%%%%% Two tangents
0&\defcolor{1}&3&3& 3& 3& 3& 3& 3& 3& 3& 3& 3& 3& 3& 3& 3& 3& 3& 3\\
1& 1&3&6&10&12&\defcolor{12}&11&11&11&11&11&11&11&11&11&11&11&11&11\\
2& 1&3&6&10&15&21&25&27&27&\defcolor{25}&23&23&23&23&23&23&23&23&23\\
3& 1&3&6&10&15&21&28&36&42&46&48&48&46&\defcolor{42}&40&40&40&40&40\\
4& 1&3&6&10&15&21&28&36&45&55&63&69&73&75&75&73&\defcolor{69}&63&63\\\hline 
\end{tabular}
}
\end{center}

\end{table}
%%%%%%%%%%%%%%%%%%%%%%%%%%%%%%%%%%%%%%%%%%%%%%%%%%%%%%%%%%%%%%%%%%%%%%%%%%%%%%%%%

We find similar behavior when the three quadrics define a scheme of multiplicity two, for us the points
$[0:0:1]$ and $[2:0:1]$.
Let 
$A:=x^2+xy+y^2-2xz$, 
$B:=2x^2+xy+2y^2-4xz-2yz$, 
$C:=x^2+xy+2y^2-2xz+6yz$, and 
$D:=x^2-xy-2y^2-2xz+2yz$.
Then $\langle A,B,C\rangle$ and $\langle A,B,D\rangle$ both define the same scheme consisting of those two
reduced points.
They have distinct tangents at $[0:0:1]$, and $A$, $B$, and $C$ have distinct tangents at $[2:0:1]$, but $B$
and $D$ are tangent at $[2:0:1]$.
Figure~\ref{F:Mult_two} shows 
%%%%%%%%%%%%%%%%%%%%%%%%%%%%%%%%%%%%%%%%%%%%%%%%%%%%%%%%%%%%%%%%%%%%%%%%%%%%%%%%%
\begin{figure}[htb]
\[
   \begin{picture}(175,125)
     \put(0,0){\includegraphics{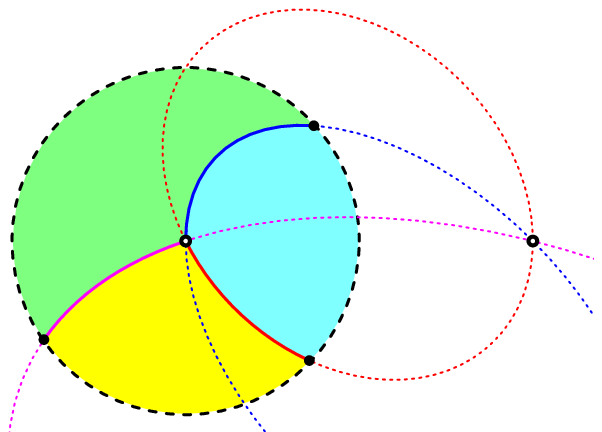}}
     \put(54,78){$A$} \put(74,36){$B$} \put(25,45){$C$}
   \end{picture}
   \quad
   \begin{picture}(175,125)
     \put(0,0){\includegraphics{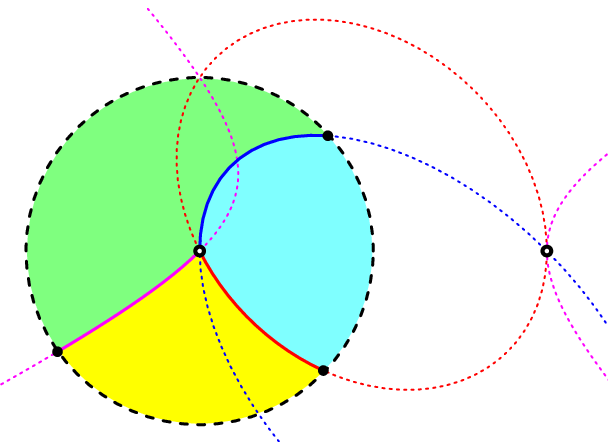}}
     \put(54,78){$A$} \put(74,36){$B$} \put(27,40){$D$}
   \end{picture}
\]
 \caption{Three conics defining two points.}
 \label{F:Mult_two}
\end{figure}
%%%%%%%%%%%%%%%%%%%%%%%%%%%%%%%%%%%%%%%%%%%%%%%%%%%%%%%%%%%%%%%%%%%%%%%%%%%%%%%%%
the resulting cell complexes and the underlying curves.
Table~\ref{Ta:ThreeQuadrics_two} shows the Hilbert functions of $S/J(\upsilon)$.

%%%%%%%%%%%%%%%%%%%%%%%%%%%%%%%%%%%%%%%%%%%%%%%%%%%%%%%%%%%%%%%%%%%%%%%%%%%%%%%%%
\begin{table}[htb]
 \caption{Hilbert functions of $S/J(\upsilon)$ for three quadrics defining two points.}
 \label{Ta:ThreeQuadrics_two}

\begin{center}
{\small
\begin{tabular}{|r|cccccccccccccccccc|}\hline
&\multicolumn{18}{c|}{$d$}\\\hline
$r$&0&1&2&3&4&5&6&7&8&9&10&11&12&13&14&15&16&17\\\hline
%%%%%%%%%   Two simple points
 0&1&3&3& 2& 2& 2& 2& 2& 2& 2& 2& 2& 2& 2&  2&  2&  2&  2\\ 
 1&1&3&6&10&12&12&10& 7& 6& 6& 6& 6& 6& 6&  6&  6&  6&  6\\ 
 2&1&3&6&10&15&21&25&27&27&25&21&16&14&14& 14& 14& 14& 14\\ 
 3&1&3&6&10&15&21&28&36&42&46&48&48&46&42& 36& 29& 25& 24\\\hline
\end{tabular}\vspace{5pt}

\begin{tabular}{|r|cccccccccccccccccc|}\hline
&\multicolumn{18}{c|}{$d$}\\\hline
$r$&0&1&2&3&4&5&6&7&8&9&10&11&12&13&14&15&16&17\\\hline
%%%%%%%%%   Two tangent at one point
 0&1&3&3& 2& 2& 2& 2& 2& 2& 2& 2& 2& 2& 2&  2&  2&  2&  2\\ 
 1&1&3&6&10&12&12&10& 7& 7& 7& 7& 7& 7& 7&  7&  7&  7&  7\\ 
 2&1&3&6&10&15&21&25&27&27&25&21&16&15&15& 15& 15& 15& 15\\ 
 3&1&3&6&10&15&21&28&36&42&46&48&48&46&42& 36& 29& 26& 26\\\hline 
\end{tabular}
}
\end{center}

%%%%%%%%%   Two simple points
%0,{1,3,3,2,2,2,2,2,2,2,2,2,2,2,2,2,2,2,2,2,2,2,2,2,2,2,2,2,2,2,2,2,2,2,2,2,2
%1,{1,3,6,10,12,12,10,7,6,6,6,6,6,6,6,6,6,6,6,6,6,6,6,6,6,6,6,6,6,6,6,6,6,6,6,6,6
%2,{1,3,6,10,15,21,25,27,27,25,21,16,14,14,14,14,14,14,14,14,14,14,14,14,14,14,14,14,14,14,14,14,14,14,14,14,14
%3,{1,3,6,10,15,21,28,36,42,46,48,48,46,42,36,29,25,24,24,24,24,24,24,24,24,24,24,24,24,24,24,24,24,24,24,24,24
%4,{1,3,6,10,15,21,28,36,45,55,63,69,73,75,75,73,69,63,55,46,40,38,38,38,38,38,38,38,38,38,38,38,38,38,38,38,38
%5,{1,3,6,10,15,21,28,36,45,55,66,78,88,96,102,106,108,108,106,102,96,88,78,67,59,55,54,54,54,54,54,54,54,54,54,54,54
%
%%%%%%%%%   Two tangent at one point
%0,{1,3,3,2,2,2,2,2,2,2,2,2,2,2,2,2,2,2,2,2,2,2,2,2,2,2,2,2,2,2,2,2,2,2,2,2,2
%1,{1,3,6,10,12,12,10,7,7,7,7,7,7,7,7,7,7,7,7,7,7,7,7,7,7,7,7,7,7,7,7,7,7,7,7,7,7
%2,{1,3,6,10,15,21,25,27,27,25,21,16,15,15,15,15,15,15,15,15,15,15,15,15,15,15,15,15,15,15,15,15,15,15,15,15,15
%3,{1,3,6,10,15,21,28,36,42,46,48,48,46,42,36,29,26,26,26,26,26,26,26,26,26,26,26,26,26,26,26,26,26,26,26,26,26
%4,{1,3,6,10,15,21,28,36,45,55,63,69,73,75,75,73,69,63,55,46,41,41,41,41,41,41,41,41,41,41,41,41,41,41,41,41,41
%5,{1,3,6,10,15,21,28,36,45,55,66,78,88,96,102,106,108,108,106,102,96,88,78,67,60,59,59,59,59,59,59,59,59,59,59,59,59

\end{table}
%%%%%%%%%%%%%%%%%%%%%%%%%%%%%%%%%%%%%%%%%%%%%%%%%%%%%%%%%%%%%%%%%%%%%%%%%%%%%%%%%@

\end{example}    

\begin{remark}
The multiplicity of a zero-dimensional scheme is the sum of its local multiplicities at each point of its
support.
For $S/J(\upsilon)$, this is
\[
\mult(S/J(\upsilon))\ =\ \sum_{\nu\in\supp(S/J(\upsilon))} \mult_\nu(S/J(\upsilon)),
\]
where $\mult_\nu(S/J(\upsilon))$ is the vector space dimension of the local ring
$(S/J(\upsilon))_{\frakm_\nu}$ with $\frakm_\nu$ the maximal ideal of the point $\nu$.  
This is the multiplicity of the tangent cone of $S/J(\upsilon)$ at $\nu$ (see ~\cite[\S~5.4]{Eisenbud}).  
Thus we should expect that we can read off the multiplicity of the schemes in Example~\ref{ex:multipoint} as
sums of local multiplicities which depend only on the geometry of the tangent cones at points in the support
of $S/J(\upsilon)$.  
This is indeed the case; to see this, we write the multiplicites of Tables~\ref{Ta:ThreeQuadrics}
and~\ref{Ta:ThreeQuadrics_two} as sums of the multiplicity in Corollary~\ref{C:multiplicity} and the
multiplicity in Table~\ref{T:CLOEx}. 
 Call the multiplicity in Corollary~\ref{C:multiplicity} the \demph{generic multiplicity}; by
 Theorem~\ref{Th:generalPowers} this is the multiplicity of $S/J(\upsilon)$ when tangents are distinct. 

In Table~\ref{Ta:ThreeQuadrics}, note that if the tangents of the edge forms at all points in the support of
$S/J(\upsilon)$ are distinct, then the multiplicity of $S/J(\upsilon)$ is thrice the generic multiplicity.  
If the tangents of edge forms are distinct at two points of support but two tangents coincide at the third
point, then the geometry at the third point is the same as in Example~\ref{ex:CLOex}.  
The multiplicity of $S/J(\upsilon)$ is twice the generic multiplicity plus the multiplicity given in
Table~\ref{T:CLOEx}.  
If tangents of edge forms are distinct at one point but two tangents coincide at both other points, then the
multiplicity of $S/J(\upsilon)$ is the generic multiplicity plus twice the multiplicity given in
Table~\ref{T:CLOEx}.  
The same observations can be made in Table~\ref{Ta:ThreeQuadrics_two}.
\end{remark}

%%%%%%%%%%%%%%%%%%%%%%%%%%%%%%%%%%%%%%%%%%%%%%%%%%%%%%%%%%%%%%%%%%%%%%%%%%%%%%%%%
\providecommand{\bysame}{\leavevmode\hbox to3em{\hrulefill}\thinspace}
\providecommand{\MR}{\relax\ifhmode\unskip\space\fi MR }
\providecommand{\MRhref}[2]{%
  \href{http://www.ams.org/mathscinet-getitem?mr=#1}{#2}
}
\providecommand{\href}[2]{#2}

%%%%%%%%%%%%%%%%%%%%%%%%%%%%%%%%%%%%%%%%%%%%%%%%%%%%%%%%%%%%%%%%%%%%%%%%%%%%%%%%%
\end{document}